\documentclass[11pt]{llncs}
\usepackage{epsfig,latexsym}
\usepackage{xcolor}
\usepackage{tikz} 
\usepackage{subcaption}
\usepackage{amssymb}

\newcommand{\WW}{$W_6^{++}$}

\newtheorem{thml}{Theorem}

\newtheorem{claimf}{Claim}

\usepackage{color}
\newcommand{\commentout}[1]{}

\newcommand{\cR}{{\cal R}}
\newcommand{\cC}{{\cal C}}

\begin{document}
\title{$\alpha_i$-Metric Graphs: Hyperbolicity
\thanks{This work was supported by a grant of the Romanian Ministry of Research, Innovation and Digitalization, CCCDI - UEFISCDI, project number PN-III-P2-2.1-PED-2021-2142, within PNCDI III.}}

\author{Feodor F. Dragan\inst{1} \and
Guillaume Ducoffe\inst{2}  }

\institute{Computer Science Department, Kent State University, Kent, USA 
\email{dragan@cs.kent.edu}  
\and 
National Institute for Research and
Development in Informatics and \\ University
of Bucharest, Bucure\c{s}ti, Rom\^{a}nia 
\email{guillaume.ducoffe@ici.ro} 
}

\maketitle

\begin{abstract}
     A graph is called $\alpha_i$-metric ($i \in \mathbb{N}$) if it satisfies the following $\alpha_i$-metric property for every vertices $u, w, v$ and $x$: if a shortest path between $u$ and $w$ and a shortest path between $x$ and $v$ share a terminal edge $vw$, then $d(u,x) \ge d(u,v) + d(v,x) - i$. The latter is a discrete relaxation of the property that in Euclidean spaces the union of two geodesics sharing a terminal segment must be also a geodesic. Recently in (Dragan \& Ducoffe, {\it WG'23}) we initiated the study of the algorithmic applications of $\alpha_i$-metric graphs. Our results in this prior work were very similar to those established in (Chepoi et al., {\it SoCG'08}) and (Chepoi et al., {\it COCOA'18}) for graphs with bounded {\em hyperbolicity}. The latter is a heavily studied metric tree-likeness parameter first introduced by Gromov. In this paper, we clarify the relationship between hyperbolicity and the $\alpha_i$-metric property, proving that $\alpha_i$-metric graphs are $f(i)$-hyperbolic for some function $f$ linear in $i$.
     We provide different proofs of this result, using various equivalent definitions to graph hyperbolicity. By contrast, we give simple constructions of $1$-hyperbolic graphs that are not $\alpha_i$-metric for any constant $i$.
     Finally, in the special case of $i=1$, we prove that $\alpha_1$-metric graphs are $1$-hyperbolic, and the bound is sharp.
     By doing so, we answer some questions left open in (Dragan \& Ducoffe, {\it WG'23}).
     \medskip

\end{abstract}

\section{Introduction}\label{sec:intro}

For any undefined graph terminology, see~\cite{BoMu08}.
In what follows, we only consider graphs $G=(V,E)$ that are finite, undirected, unweighted, simple and connected.
The \emph{distance} $d_G(u,v)$ between two vertices $u,v \in V$ is the minimum length (number of edges) of a path between $u$ and $v$ in $G$. 
The \emph{interval} $I_G(u,v)$ between $u$ and $v$ is the set of all vertices contained in some shortest $(u,v)$-path, {\it i.e.}, $I_G(u,v) = \{ w \in V : d_G(u,v) = d_G(u,w) + d_G(w,v)\}$.
We compare two families of metric properties on graphs, defined on  quadruples of vertices $u,v,w,x$ as follows for every $i \in \mathbb{N}$:
\begin{description} 
\item[$\alpha_i$-metric property:] if $v \in I_G(u,w)$ and $w\in I_G(v,x)$ are adjacent,
  then \\  $~~~~~~~~~~~~~~~~~~~~~~~~~~~d_G(u,x)\geq d_G(u,v) + d_G(v,x)-i=d_G(u,v) + 1 + d_G(w,x)-i;$
\end{description} 
\begin{description} 
\item[$\beta_i$-metric property:] if $d_G(u,v) + d_G(w,x) \le d_G(u,w) + d_G(v,x) \le d_G(u,x) + d_G(v,w)$,
  then \\  $~~~~~~~~~~~~~~~~~~~~~~~~~~~d_G(u,x) + d_G(v,w) - d_G(u,w) - d_G(v,x) \le i.$
\end{description} 
A graph is called \emph{$\alpha_i$-metric} if it satisfies the $\alpha_i$-metric property for every four vertices $u,v,w,x$.
It is called \emph{$\delta$-hyperbolic} if it satisfies the $\beta_{2\delta}$-metric property for every four vertices $u,v,w,x$.
The \emph{hyperbolicity} of $G$ is the smallest half-integer $\delta$ such that it is $\delta$-hyperbolic.
Roughly, in an $\alpha_i$-metric graph $G$, the union of two shortest paths with a common terminal edge results in an ``almost shortest'' path with defect at most $i$.
By comparison, the less intuitive parameter hyperbolicity represents the smallest possible $\delta$ such that any four vertices of $G$ can be embedded in a weighted tree with distortion at most $\delta$ of their pairwise distances in $G$~\cite{Dres84}\footnote{
Let $G=(V,E)$ be a graph, and let $X \subseteq V$.
An \emph{embedding} from $X$ to another graph $H=(V',E')$ is a mapping $\varphi : X \mapsto V'$.
The \emph{distortion} of the embedding $\varphi$ is defined as $\max\{ |d_G(x,y)-d_H(\varphi(x),\varphi(y))| : x,y \in X \}$.
We call $\varphi$ an \emph{isometric} embedding if it has null distortion: $d_G(x,y) = d_H(\varphi(x),\varphi(y))$ for every $x,y \in X$. 
}.
Both the $\alpha_0$-metric property and the $\beta_0$-metric property are satisfied by trees.
In~\cite{BaCh2003}, it was observed that if there exists an embedding of a graph $G$ in some weighted tree $T$ with distortion at most $\eta$, then $G$ must be $\alpha_{\lfloor 3\eta \rfloor}$-metric and $\frac{\lfloor 4\eta \rfloor}2$-hyperbolic. This makes these properties interesting to study on their own.

    \paragraph{Related work.}

The general $\alpha_i$-metric property ($i \in \mathbb{N}$) was introduced in~\cite{Ch86}, then further studied in~\cite{YuCh1991}.
Furthermore, the $\alpha_1$-metric property was considered in~\cite{Ch86,Ch88} for chordal graphs.
It was proved that all chordal graphs~\cite{Ch86} and all plane triangulations with inner vertices of degree at least seven~\cite{WG16} are $\alpha_1$-metric.
All distance-hereditary graphs~\cite{YuCh1991}, and even more strongly, all HHD-free graphs~\cite{ChD03}, are $\alpha_2$-metric. 
The $\alpha_0$-metric graphs are exactly the ptolemaic graphs, {\it i.e.} the chordal distance-hereditary graphs~\cite{Hov77}.
Furthermore, a characterization of $\alpha_1$-metric graphs was given in~\cite{YuCh1991}: they are exactly the graphs with convex disks and one forbidden isometric subgraph \WW (see Fig.~\ref{fig:forbid}).
More recently, additional properties of $\alpha_1$-metric graphs and $\alpha_i$-metric graphs ($i \in \mathbb{N}$) were reported in~\cite{DrDu23,DrGu21,WG16}.
In~\cite{DrDu23}, we presented the first algorithmic applications of $\alpha_i$-metric graphs to classical distance problems such as diameter, radius and all eccentricities computations.
More specifically, all vertex eccentricities in an $\alpha_i$-metric graph can be approximated in linear time up to some additive term in ${\cal O}(i)$.
Furthermore, there exists a subquadratic-time algorithm for the exact computation of the radius of an $\alpha_1$-metric graph.

Gromov's notion of hyperbolicity, see~\cite{Gromov}, was arguably the start of Geometric Group Theory.
Since then, the study of $\delta$-hyperbolic graphs has become also an important topic in Metric Graph Theory~\cite{BaCh08}.
This parameter has attracted further attention in Network Science, both as a way to better classify complex networks~\cite{AbDr16,KSN16} and to explain some of their properties such as core congestion~\cite{ChDrVa17}. Many real-world networks have small hyperbolicity~\cite{AbDr16,KSN16,Vien}.  
See also~\cite{ADM14,BoChCa15,slimness,JoLo04} for other related results on hyperbolicity. 
Many different approaches have been proposed in order to upper bound the hyperbolicity in some graph classes~\cite{DrGu19,KoMo02,WuZh01}. 
In particular, chordal graphs are $1$-hyperbolic, and the chordal graphs with hyperbolicity strictly less than one can be characterized with two forbidden isometric subgraphs~\cite{BKM01}.
The $0$-hyperbolic graphs are exactly the block graphs, {\it i.e.} the graphs such that every biconnected component is a clique~\cite{Hov79}.
Characterizations of the $\frac 1 2$-hyperbolic graphs were given in~\cite{BaCh2003,CoDu14}.
Furthermore, the algorithmic applications of $\delta$-hyperbolic graphs have been studied much earlier than for $\alpha_i$-metric graphs~\cite{CDEHVX12,ChDEHV08,Chepoi2018FastEA}.
In~\cite{ChDEHV08,Chepoi2018FastEA} it was proved that all vertex eccentricities in a $\delta$-hyperbolic graph can be approximated in linear time up to some additive term in ${\cal O}(\delta)$.
Therefore, it can be argued that $\alpha_i$-metric graphs and $\delta$-hyperbolic graphs have very similar algorithmic properties.

Little is known about the relationships between $\alpha_i$-metric properties and hyperbolicity.
On one hand, the authors of~\cite{BaCh2003} observed that every $0$-hyperbolic graph must be $\alpha_0$-metric, and every $\frac 1 2$-hyperbolic graph must be $\alpha_1$-metric.
However, for any positive integer $i$, the ladder graph of height $\left\lfloor i/2\right\rfloor+1$, see Fig.~\ref{fig:ladder-c4}, is a simple example of a $1$-hyperbolic graph that is not $\alpha_i$-metric (we give other simple constructions in Section~\ref{sec:obstructions-alphai}, where we further discuss the relation between the $\alpha_i$-metric property in some classes of graphs of small hyperbolicity and the existence of isometric ladders of large height).
On the other hand, the authors of~\cite{ChChChJa22+} briefly discussed the  hyperbolicity and the $\alpha_i$-metric property for geodesic metric spaces. 
They observed that Euclidean spaces must be $\alpha_0$-metric (that is because the union of two geodesics that overlap  on one terminal segment is also a geodesic) whereas they have unbounded hyperbolicity.
However, they also noted that ``for graphs, the links between $\delta$-hyperbolic graphs and graphs with $\alpha_i$-metrics are less clear''. 

    \paragraph{Our Contributions.}

The main result in the paper is that $\alpha_i$-metric graphs are $f(i)$-hyperbolic, for some function $f$ that only depends on $i$. 
Throughout Section~\ref{sec:gal-i} we give different proofs of this result. The strongest result is proven in   Section~\ref{sec:helly}. 
\begin{thml}[see Theorem~\ref{thm:main} in the paper]
    Every $\alpha_i$-metric graph is $3(i+1)/2$-hyperbolic.
\end{thml}
In Sections~\ref{sec:papasoglu}--\ref{sec:triangles}, we compare the $\alpha_i$-metric property with some other graph parameters that are different, but functionally equivalent, to hyperbolicity.
By doing so, we obtain arguably simpler proofs, and sharper relations between the $\alpha_i$-metric property and these other parameters than by using their known relationships with hyperbolicity. 
Conversely, partial characterizations of the $\alpha_i$-metric property for graphs with small hyperbolicity are reported in Section~\ref{sec:obstructions-alphai}. 
Our current best estimate on function $f$ is linear, but it is probably not tight. 
We conjecture the right upper bound to be $f(i) = \frac{i+1}{2}$, that would be sharp.
In Section~\ref{sec:i-1}, we prove this conjecture if $i \le 1$. 
\begin{thml}[see Theorem~\ref{thm:alpha-1} in the paper]
    Every $\alpha_1$-metric graph is $1$-hyperbolic.    
\end{thml}
Some consequences of our results are further discussed in Section~\ref{sec:i-1}. 
Some preliminary results are recalled in Section~\ref{sec:prelim}.
We conclude this paper in Section~\ref{sec:ccl}, where we introduce an intriguing generalization of both $\alpha_i$-metric graphs and $\delta$-hyperbolic graphs. 

\section{Preliminaries}\label{sec:prelim}

In what follows, for every two vertices $u$ and $v$ in a graph $G=(V,E)$, we write $u \sim v$ ($u \not\sim v$, respectively) if and only if $u$ and $v$ are adjacent (nonadjacent, respectively).
For every vertex $v$ in a graph $G=(V,E)$, its open and closed \emph{neighbourhoods} are defined as $N_G(v) = \{u \in V: u \sim v\}$ and $N_G[v] = N_G(v) \cup \{v\}$.
The \emph{disk} of center $v$ and radius $r$ is defined as $D_G(v,r) = \{u \in V : d_G(u,v) \le r\}$.
The following variation of Lemma $4$ in~\cite{DrDu23} is used in our proofs:

\begin{lemma}\label{lem:three-balls}
Let $D_G(u,r_u),D_G(v,r_v),D_G(w,r_w)$ be pairwise intersecting disks of $G$.
If $G$ is $\alpha_i$-metric, then there exists a vertex $x$ such that $d_G(u,x) \leq r_u$, $d_G(v,x) \leq r_v$ and $d_G(w,x) \leq r_w + i$.
\end{lemma}
\begin{proof}
Let $x \in D_G(u,r_u) \cap D_G(v,r_v)$ be such that $d_G(w,x)$ is minimized.
Suppose, by way of contradiction, that  $d_G(x,w) > r_w + i$, and let $y \in N(x)$ be on a shortest $(x,w)$-path.
By minimality of $d_G(x,w)$, we must have $d_G(u,y) > r_u$ or $d_G(v,y) > r_v$.
By symmetry, we may assume that $d_G(u,y) > r_u$.
Since $r_u = d_G(u,x) < d_G(u,y)$, $d_G(w,y) < d_G(w,x)$, and $x \sim y$, by the $\alpha_i$-metric property applied to $u,x,y,w$ it implies $d_G(u,w) \geq d_G(u,x) + d_G(x,w) - i > r_u + (r_w+i) - i = r_u+r_w \ge d_G(u,w)$, giving a contradiction.\qed
\end{proof}

Let $u$ and $v$ be arbitrary vertices in $G$.
Recall that $I_G(u,v)$ contains every vertex on a shortest $(u,v)$-path.
Let also $I_G^o(u,v)=I_G(u,v)\setminus \{u,v\}$.
A set of vertices $S$ is called \emph{convex} if $I_G(x,y) \subseteq S$ for any two vertices $x,y \in S$.
For every integer $k$ such that $0 \le k \le d_G(u,v)$, we define the \emph{slice} $S_k(u,v,G) = \{x \in I_G(u,v) : d_G(u,x) = k\}$.
Let $\kappa_G(u,v)$ denote the maximum diameter of a slice between $u$ and $v$: $\kappa_G(u,v) = \max_{0 \le k \le d_G(u,v)}\max\{d_G(x,y): x,y \in S_k(u,v,G)\}$.
We define the \emph{interval thinness} of $G$ as $\kappa(G) = \max\{\kappa_G(u,v) : u,v \in V\}$.

\begin{lemma}[see Lemma 1 in~\cite{DrDu23}]\label{lm:auxiliary-GD}
Let $G$ be an $\alpha_i$-metric graph, and let $u,v,x,y$ be vertices such that $x\in I_G(u,v)$, $d_G(u,x)=d_G(u,y)$, and $d_G(v,y)\le d_G(v,x)+k$ for some nonnegative integer $k$. Then, $d_G(x,y)\le k+i+2$.
\end{lemma}

\begin{lemma}[see Lemma 2 in~\cite{DrDu23}]\label{lem:thinness}
If $G=(V,E)$ is an $\alpha_i$-metric graph, then its interval thinness is at most $i+1$.
\end{lemma}

The bound of Lemma~\ref{lem:thinness} is sharp, as shown in Fig.~\ref{fig:leanness-tight-bound}.

\medskip
A \emph{geodesic triangle} $\Delta_G(u,v,w) = P_G(u,v) \cup P_G(v,w) \cup P_G(w,u)$ is the union of a shortest $(u,v)$-path $P_G(u,v)$, a shortest $(v,w)$-path $P_G(v,w)$ and a shortest $(w,u)$-path $P_G(w,u)$. Note that $P_G(u,v),P_G(v,w),P_G(w,u)$ are called the \emph{sides} of the triangle, and they may not be disjoint. 

A \emph{metric triangle} is a triple $u v w$ such that $I_G^o(u,v)$, $I_G^o(v,w)$ and $I_G^o(w,u)$ are pairwise disjoint.
The \emph{type} of this triangle is the triple $(d_G(u,v),d_G(v,w),d_G(w,u))$.
Let also $\max\{d_G(u,v),d_G(v,w),d_G(w,u)\}$ be the \emph{maximum side-length} of the triangle. 
Finally, given a triple $u,v,w \in V$, a \emph{quasi-median} is a metric triangle $u'v'w'$ such that we have:
\begin{itemize}
    \item $d_G(u,v) = d_G(u,u') + d_G(u',v') + d_G(v',v)$; 
    \item $d_G(v,w) = d_G(v,v') + d_G(v',w') + d_G(w',w)$;
    \item $d_G(w,u) = d_G(w,w') + d_G(w',u') + d_G(u',u)$.
\end{itemize}
In any connected graph, every triple has at least one quasi median ({\it e.g.}, see~\cite{ChDEHV08}).

\medskip
Other notations and terminology are locally defined at appropriate places throughout the paper.
Furthermore, if $G$ is clear from the context then it is removed from all our notations.

\begin{figure}[!h]
    \centering
    \resizebox{2.5in}{!}{
    \begin{tikzpicture}
        \node[circle,fill=black,inner sep=0pt,minimum size=5pt, label=left:{$u$}] at (-4,0) {};
        \node[circle,fill=black,inner sep=0pt,minimum size=5pt, label=right:{$v$}] at (4,0) {};
        \node[circle,fill=black,inner sep=0pt,minimum size=5pt, label=above:{$x$}] at (0,2) {};
        \node[circle,fill=black,inner sep=0pt,minimum size=5pt, label=below:{$y$}] at (0,-2) {};
        \node[circle,fill=black,inner sep=0pt,minimum size=5pt] at (-3.5,.25) {};
        \node[circle,fill=black,inner sep=0pt,minimum size=5pt] at (-3.5,-.25) {};
        \node[circle,fill=black,inner sep=0pt,minimum size=5pt] at (-3,.5) {};
        \node[circle,fill=black,inner sep=0pt,minimum size=5pt] at (-3,0) {};
        \node[circle,fill=black,inner sep=0pt,minimum size=5pt] at (-3,-.5) {};
        \node[circle,fill=black,inner sep=0pt,minimum size=5pt] at (-2.5,.75) {};
        \node[circle,fill=black,inner sep=0pt,minimum size=5pt] at (-2.5,.25) {};
        \node[circle,fill=black,inner sep=0pt,minimum size=5pt] at (-2.5,-.25) {};
        \node[circle,fill=black,inner sep=0pt,minimum size=5pt] at (-2.5,-.75) {};
        \node[circle,fill=black,inner sep=0pt,minimum size=5pt] at (-2,1) {};
        \node[circle,fill=black,inner sep=0pt,minimum size=5pt] at (-2,.5) {};
        \node[circle,fill=black,inner sep=0pt,minimum size=5pt] at (-2,0) {};
        \node[circle,fill=black,inner sep=0pt,minimum size=5pt] at (-2,-.5) {};
        \node[circle,fill=black,inner sep=0pt,minimum size=5pt] at (-2,-1) {};
        \node[circle,fill=black,inner sep=0pt,minimum size=5pt] at (3.5,.25) {};
        \node[circle,fill=black,inner sep=0pt,minimum size=5pt] at (3.5,-.25) {};
        \node[circle,fill=black,inner sep=0pt,minimum size=5pt] at (3,.5) {};
        \node[circle,fill=black,inner sep=0pt,minimum size=5pt] at (3,0) {};
        \node[circle,fill=black,inner sep=0pt,minimum size=5pt] at (3,-.5) {};
        \node[circle,fill=black,inner sep=0pt,minimum size=5pt] at (2.5,.75) {};
        \node[circle,fill=black,inner sep=0pt,minimum size=5pt] at (2.5,.25) {};
        \node[circle,fill=black,inner sep=0pt,minimum size=5pt] at (2.5,-.25) {};
        \node[circle,fill=black,inner sep=0pt,minimum size=5pt] at (2.5,-.75) {};
        \node[circle,fill=black,inner sep=0pt,minimum size=5pt] at (2,1) {};
        \node[circle,fill=black,inner sep=0pt,minimum size=5pt] at (2,.5) {};
        \node[circle,fill=black,inner sep=0pt,minimum size=5pt] at (2,0) {};
        \node[circle,fill=black,inner sep=0pt,minimum size=5pt] at (2,-.5) {};
        \node[circle,fill=black,inner sep=0pt,minimum size=5pt] at (2,-1) {};
        \node[circle,fill=black,inner sep=0pt,minimum size=5pt] at (0,1.5) {};
        \node[circle,fill=black,inner sep=0pt,minimum size=5pt] at (0,1) {};
        \node[circle,fill=black,inner sep=0pt,minimum size=5pt] at (0,.5) {};
        \node[circle,fill=black,inner sep=0pt,minimum size=5pt] at (0,-1.5) {};
        \node[circle,fill=black,inner sep=0pt,minimum size=5pt] at (0,-1) {};
        \node[circle,fill=black,inner sep=0pt,minimum size=5pt] at (-.5,-1.75) {};
        \node[circle,fill=black,inner sep=0pt,minimum size=5pt] at (-.5,-1.25) {};
        \node[circle,fill=black,inner sep=0pt,minimum size=5pt] at (-.5,-.75) {};
        \node[circle,fill=black,inner sep=0pt,minimum size=5pt] at (.5,1.75) {};
        \node[circle,fill=black,inner sep=0pt,minimum size=5pt] at (.5,1.25) {};
        \node[circle,fill=black,inner sep=0pt,minimum size=5pt] at (.5,.75) {};
        \node[circle,fill=black,inner sep=0pt,minimum size=5pt] at (-.5,1.75) {};
        \node[circle,fill=black,inner sep=0pt,minimum size=5pt] at (-.5,1.25) {};
        \node[circle,fill=black,inner sep=0pt,minimum size=5pt] at (.5,-1.75) {};
        \node[circle,fill=black,inner sep=0pt,minimum size=5pt] at (.5,-1.25) {};
        \draw (-4,0) -- (-2,1); \draw[dashed] (-2,1) -- (-.5,1.75); \draw (-.5,1.75) -- (0,2);
        \draw (-4,0) -- (-2,-1); \draw[dashed] (-2,-1) -- (-.5,-1.75); \draw (-.5,-1.75) -- (0,-2);
        \draw (4,0) -- (2,1); \draw[dashed] (2,1) -- (.5,1.75); \draw (.5,1.75) -- (0,2);
        \draw (4,0) -- (2,-1); \draw[dashed] (2,-1) -- (.5,-1.75); \draw (.5,-1.75) -- (0,-2);
        \draw (-3.5,-.25) -- (-3.5,.25); \draw (-3,-.5) -- (-3,.5); \draw (-2.5,-.75) -- (-2.5,.75); \draw (-2,-1) -- (-2,1);
        \draw (3.5,-.25) -- (3.5,.25); \draw (3,-.5) -- (3,.5); \draw (2.5,-.75) -- (2.5,.75); \draw (2,-1) -- (2,1);
        \draw (-.5,1.75) -- (-.5,1.25); \draw[dashed] (-.5,1.25) -- (-.5,-.75); \draw (-.5,-.75) -- (-.5,-1.75);
        \draw (.5,1.75) -- (.5,.75); \draw[dashed] (.5,.75) -- (.5,-1.25); \draw (.5,-1.25) -- (.5,-1.75);
        \draw (0,2) -- (0,.5); \draw[dashed] (0,.5) -- (0,-1); \draw (0,-1) -- (0,-2);
        \draw (-3.5,-.25) -- (-2,.5); \draw[dashed] (-2,.5) -- (-.5,1.25); \draw (-.5,1.25) -- (.5,1.75);
        \draw (-3,-.5) -- (-2,0); \draw[dashed] (-2,0) -- (0,1); \draw (0,1) -- (.5,1.25); \draw[dashed] (.5,1.25) -- (1,1.5);
        \draw (-2.5,-.75) -- (-2,-.5); \draw[dashed] (-2,-.5) -- (0,.5); \draw (0,.5) -- (.5,.75); \draw[dashed] (.5,.75) -- (1.5,1.25);
        \draw[dashed] (-2,-1) -- (2,1); 
        \draw[dashed] (-1.5,-1.25) -- (2,.5); \draw (2,.5) -- (2.5,.75);
        \draw[dashed] (-1,-1.5) -- (-.5,-1.25); \draw (-.5,-1.25) -- (0,-1); \draw[dashed] (0,-1) -- (2,0); \draw (2,0) -- (3,.5);
        \draw (-.5,-1.75) -- (.5,-1.25); \draw[dashed] (.5,-1.25) -- (2,-.5); \draw (2,-.5) -- (3.5,.25);
        \draw (3.5,-.25) -- (2,.5); \draw[dashed] (2,.5) -- (.5,1.25); \draw (.5,1.25) -- (-.5,1.75);
        \draw (3,-.5) -- (2,0); \draw[dashed] (2,0) -- (.5,.75); \draw (.5,.75) -- (-.5,1.25); \draw[dashed] (-.5,1.25) -- (-1,1.5);
        \draw (2.5,-.75) -- (2,-.5); \draw[dashed] (2,-.5) -- (0,.5); \draw (0,.5) -- (-.5,.75); \draw[dashed] (-.5,.75) -- (-1.5,1.25);
        \draw[dashed] (2,-1) -- (-2,1); 
        \draw[dashed] (1.5,-1.25) -- (-2,.5); \draw (-2,.5) -- (-2.5,.75);
        \draw[dashed] (1,-1.5) -- (.5,-1.25); \draw (.5,-1.25) -- (-.5,-.75); \draw[dashed] (-.5,-.75) -- (-2,0); \draw (-2,0) -- (-3,.5);
        \draw (.5,-1.75) -- (-.5,-1.25); \draw[dashed] (-.5,-1.25) -- (-2,-.5); \draw (-2,-.5) -- (-3.5,.25);
    \end{tikzpicture}}
    \vspace*{-2mm}
    \caption{The triangular $(i+2,i+2)$-grid is $\alpha_i$-metric. As $x,y \in S_{i+1}(u,v)$, its interval thinness is at least $d(x,y) = i+1$.}
    \label{fig:leanness-tight-bound}
    \vspace*{-10mm}
\end{figure}

\section{Hyperbolicity of $\alpha_i$-metric graphs for arbitrary $i \ge 0$}\label{sec:gal-i}

Our main result in this section is that $\alpha_i$-metric graphs are $\frac 3 2(i+1)$-hyperbolic.
This result is reported in Section~\ref{sec:helly}.
In Sections~\ref{sec:papasoglu},~\ref{sec:cop-robber}, and~\ref{sec:triangles}, we prove bounds on some parameters that are related to hyperbolicity.
Our bounds are better than the ones that follow directly from the known relations between these parameters and hyperbolicity; see Table~\ref{tab:bounds}.
Furthermore, we present some properties of $\alpha_i$-metric graphs that do not follow from their bounded hyperbolicity. 
Finally, some results on the $\alpha_i$-metric property for graphs with small hyperbolicity are discussed in Section~\ref{sec:obstructions-alphai}. 

\begin{table}[!h]
    \vspace*{-3mm}
    \centering
    \resizebox{6in}{!}{
    \begin{tabular}{|c||c|c|c|c| }
        \hline
        Parameter & Bound from hyperbolicity $\delta$ & Application to $\delta = \frac 3 2 (i+1)$ & \textcolor{red}{Our bound} & \textcolor{red}{Reference} \\
        \hline
        \hline 
        &&&&\\
        Interval thinness of $\Sigma(G)$ & $\le 4\delta+4$ (folklore) & $6i+10$ & \textcolor{red}{$2i+12$} & \textcolor{red}{Lemma~\ref{lem:thinness-subdivision}}   \\ 
        &&&&\\
        \hline 
        &&&&\\
        $(s, s')^*$-dismantlability &  $s' = \lceil{s/2}\rceil + 2\delta$~\cite{ChChNiVa11}& $s' = \lceil{s/2}\rceil + 3i+3$ & \textcolor{red}{$s' = \lceil{s/2}\rceil + 2i+1$} & \textcolor{red}{Lemma~\ref{(s,s')-dism}} \\
        &&&&\\
        \hline 
        &&&&\\
        Thinness of geodesic triangles & $\le 4\delta$~\cite{Fast-appr-hyperb} & $6(i+1)$ & \textcolor{red}{$3(i+1)$} & \textcolor{red}{Theorem~\ref{thm:thinness}} \\
        &&&&\\
        \hline
    \end{tabular}}
    \vspace*{2mm}
    \caption{Bounds on some parameters for $\alpha_i$-metric graphs.
    }
    \label{tab:bounds}
    \vspace*{-5mm}
\end{table}

    \subsection{Using Interval Thinness}\label{sec:papasoglu}

It is easy to prove that every $\delta$-hyperbolic graph has interval thinness at most $2\delta$.
Conversely, odd cycles are examples of graphs with interval thinness equal to zero (they are so-called ``geodetic graphs'') but unbounded hyperbolicity.
Therefore, bounded interval thinness does not imply bounded hyperbolicity.
However, let the \emph{$1$-subdivision graph} $\Sigma(G)$ of a graph $G$ be obtained by replacing each edge $e=uv$ by a path $[u,e,v]$ of length two.
Papasoglu~\cite{Pap95} proved that the hyperbolicity of $G$ is at most doubly exponential in the interval thinness of $\Sigma(G)$.

By Lemma~\ref{lem:thinness}, the interval thinness of $\alpha_i$-metric graphs is at most $i+1$.
Therefore, in order to prove that $\alpha_i$-metric graphs are $f(i)$-hyperbolic, for some arbitrary $f$, it would be sufficient to prove that their $1$-subdivision graphs are $\alpha_{g(i)}$-metric, for some $g$.
Unfortunately, this is false, even for $i=1$, as shown in Fig.~\ref{fig:unbounded-alpha-i}.
Our counter-example is a triangular $(2,n)$-grid graph $G$.
This graph is $\alpha_1$-metric (that follows from the characterization proved in~\cite{YuCh1991}, this result is restated in Theorem~\ref{th:charact}).
However, with the 4-tuple $u,e,w,x$ we have $e \in I_{\Sigma(G)}(u,w)$ and $w \in I_{\Sigma(G)}(e,x)$, $d_{\Sigma(G)}(u,e) = 2n-1$, $d_{\Sigma(G)}(e,x) = 2n-1$, but $d_{\Sigma(G)}(u,x) = 2$. 
Hence, for $\Sigma(G)$ to be $\alpha_i$-metric, $i$ should be at least $d_{\Sigma(G)}(u,e) + d_{\Sigma(G)}(e,x) - d_{\Sigma(G)}(u,x) = 4n-4$.

\begin{figure}[!ht]
    \centering
    \begin{tikzpicture}

        \node[circle,fill=black,inner sep=0pt,minimum size=5pt, label=above:{$u$}] at (-1,0) {};
        \node[circle,fill=black,inner sep=0pt,minimum size=5pt, label=above:{$x$}] at (1,0) {};

        \node[circle,fill=black,inner sep=0pt,minimum size=5pt, label=above:{}] at (-1,-1) {};
        \node[circle,fill=black,inner sep=0pt,minimum size=5pt, label=above:{}] at (1,-1) {};

        \node[circle,fill=black,inner sep=0pt,minimum size=5pt, label=above:{}] at (-1,-3) {};
        \node[circle,fill=black,inner sep=0pt,minimum size=5pt, label=above:{}] at (1,-3) {};
        \node[circle,fill=black,inner sep=0pt,minimum size=5pt, label=above:{}] at (-1,-4) {};
        \node[circle,fill=black,inner sep=0pt,minimum size=5pt, label=above:{}] at (1,-4) {};
        \node[circle,fill=black,inner sep=0pt,minimum size=5pt, label=below:{$v$}] at (-1,-5) {};
        \node[circle,fill=black,inner sep=0pt,minimum size=5pt, label=below:{$w$}] at (1,-5) {};

        \node[label=below:{$e=vw$}] at (0,-5) {};

        \draw (-1,0) -- (1,0) -- (1,-1) -- (-1,-1) -- (-1,0);
        \draw (-1,-1) -- (1,0);
        \draw[dashed] (-1,-1) -- (-1,-3);
        \draw[dashed] (1,-1) -- (1,-3);
        \draw (1,-3) -- (1,-5) -- (-1,-5) -- (-1,-3) -- (1,-3); 
        \draw (1,-3) -- (-1,-4) -- (1,-4) -- (-1,-5);

    \end{tikzpicture}
    \caption{The triangular $(2,n)$-grid is $\alpha_1$-metric, but its $1$-subdivision graph is not $\alpha_i$-metric for any $i \le 4n-5$.}
    \label{fig:unbounded-alpha-i}
    \vspace*{-5mm}
\end{figure}

We now give a direct proof that $1$-subdivision graphs of $\alpha_i$-metric graphs have interval thinness at most linear in $i$.

\begin{lemma}\label{lem:thinness-subdivision}
    If $G=(V,E)$ is an $\alpha_i$-metric graph, then $\kappa(\Sigma(G)) \le 2i+12$.
\end{lemma}
\begin{proof}
    Let $s,t \in V \cup E$ be arbitrary.
    Since $\Sigma(G)$ is bipartite with partite sets $V$ and $E$, for every two consecutive slices $S_k(s,t)$ and $S_{k+1}(s,t)$, there is one slice fully in $V$ and one slice fully in $E$.
    In particular, $\kappa(s,t)$ is at most two units more than the maximum diameter of a slice $S_k(s,t) \subseteq V$.
    In what follows, we consider some vertices $p,q \in S_k(s,t) \subseteq V$, for some $k \le d_{\Sigma(G)}(s,t)$ such that $d_{\Sigma(G)}(p,q)$ is maximized.
    Let $\psi(s)$ be defined as follows: if $s \in V$ then $\psi(s) = \{s\}$, and if $s = uv \in E$ then $\psi(s) = \{u,v\}$.
    Let $\psi(t)$ be defined in a similar fashion.
    Consider some shortest $(s,t)$-paths $P,Q$ of $\Sigma(G)$ such that $p \in V(P), \ q \in V(Q)$.
    There exist some $u \in \psi(s)$ and $w \in \psi(t)$ such that $P \cap V$ induces a shortest $(u,w)$-path of $G$.
    In the same way, there exist some $u' \in \psi(s)$ and $w' \in \psi(t)$ such that $Q \cap V$ induces a shortest $(u',w')$-path of $G$.
    As in $G$ either $u=u'$ or $u \sim u'$, and similarly either $w=w'$ or $w \sim w'$, we obtain that $q$ lies on some $(u,w)$-path (not necessarily shortest) in $G$ of length at most $d_G(u,w)+2$.
    Furthermore, $d_G(u,p) \le d_G(u,q) \le d_G(u,p)+1$.
    Let $q' \in N_G[q]$ be satisfying $d_G(u,q') = d(u,p)$ (either $q'=q$, or $q' \in S_1(q,u)$).
    As $p \in I_G(u,w)$, $d_G(u,p) = d_G(u,q')$, and $d_G(q',w) \le d_G(p,w)+2$, by Lemma~\ref{lm:auxiliary-GD} applied to $u,w,p,q'$ and $k=2$, $d_G(p,q) \le d_G(p,q')+1 \le (i + 2 + 2) + 1 = i+5$.
    Therefore, $\kappa(s,t) \le d_{\Sigma(G)}(p,q) + 2 \le 2i+12$. \qed
\end{proof}

\begin{corollary}
    If $G=(V,E)$ is $\alpha_i$-metric, then it is $f(i)$-hyperbolic for some doubly exponential function $f$.
\end{corollary}

    \subsection{Using Cop-Robber Games}\label{sec:cop-robber}

An $n$-vertex graph $G$, with $n>1$, is {\em dismantlable} if its vertices can be ordered $v_1, \dots, v_n$ so that for each vertex $v_k$, $1 \le k < n$, there exists another vertex $v_{\ell}$ with
$\ell > k$, such that $N[v_k] \cap V_k \subseteq N[v_{\ell}]$, where $V_k := \{v_k, v_{k+1}, \dots , v_n\}$. Such a vertex  ordering is called a {\em dismantling ordering}. It is known~\cite{NW83,Quil85} that the dismantlable graphs are exactly those graphs where in the classical cop and robber game the cop has a winning strategy. In this classical cop and robber game, two players, the cop $\cC$ and the robber $\cR$, move
alternatively along edges of a graph $G$. The cop captures the robber if both players
are on the same vertex at the same moment of time. A graph $G$ is called {\em cop-win} if the cop always
has a strategy to capture the robber after a finite number of steps. Nowakowski, Winkler~\cite{NW83} and Quilliot~\cite{Quil85} 
characterized the cop-win graphs as graphs admitting a dismantling ordering. 

First we show that every $(i+1)^{st}$ power of an  $\alpha_i$-metric graph $G$ is dismantlable. Recall that the $\lambda^{th}$ power of a graph $G=(V,E)$ is the graph $G^{\lambda}=(V,E')$ such that $uv\in E'$ if and only if $0<d(u,v)\le \lambda$. 
For the case when $i=1$, this result follows also from \cite{chalopin2022graphs,YuCh1991}. 

\begin{lemma}\label{lm:dismantl-i}
If $G$ is an $\alpha_i$-metric graph, then $G^{i+1}$ is  dismantlable. In particular, every BFS-ordering of $G$ is a dismantling ordering of  $G^{i+1}$.  
\end{lemma}

\begin{proof} 
Let $v_1, \dots, v_n$ with $v_n:=u$ be a $BFS(u)$-ordering of $G$ started at an arbitrary vertex $u$. We will show that for every vertex $v_k$ there is a vertex $v_{\ell}$ with $\ell>k$ such that for every vertex $x\in \{v_{k+1}, \dots, v_n\}$, $d(x,v_k)\le i+1$ implies $d(x,v_{\ell})\le i+1$.  As a vertex $v_{\ell}$  we will choose any neighbour of $v_k$ in $I(v_k,v_n)$. Assume that for some $x\in \{v_{k+1}, \dots, v_n\}$ with $d(x,v_k)\le i+1$, $d(x,v_{\ell})> i+1$ holds. Then,  necessarily, $d(x,v_{\ell})= i+2$ and $d(x,v_k)= i+1$. By the $\alpha_i$-metric property applied to $v_n,v_{\ell},v_k,x$, $d(x,v_n)\ge d(v_n,v_k)+d(v_k,x)-i= d(v_n,v_k)+1$ must hold, contradicting with the fact that in any $BFS(u)$-ordering, $d(u,v_k)\ge d(u,x)$ for every $x\in \{v_{k+1}, \dots, v_n\}$. 
\qed
\end{proof}

%
%
Recently, several other variants of the classical
cop and robber game were introduced and investigated  \cite{ChChNiVa11,ChChPaPe14,FoGoKr08,FoGoKr++,NiSu08}. In a  most general extension of the game, the cop $\cC$ and the robber $\cR$ move at speeds $s'\ge 1$ and $s \ge 1$, respectively. This game was 
introduced and thoroughly investigated in \cite{ChChNiVa11,ChChPaPe14}. It generalizes the cop and fast robber game from  \cite{FoGoKr08,FoGoKr++,NiSu08}. The unique difference between this “$(s, s')$-cop and robber game” and
the classical cop and robber game is that at each step, $\cC$ can move along a path of length at most $s'$ and $\cR$ can move along a path of length at most $s$ not traversing the position occupied by the cop. In \cite{ChChNiVa11}, the class of cop-win graphs for this game was denoted by ${\cal CWFR}(s, s')$. 

Similar to the characterization of classical cop-win graphs given in~\cite{NW83,Quil85}, the $(s, s')$-cop-win graphs have been characterized in \cite{ChChNiVa11}  via a special $(s, s')$-dismantling scheme. A graph $G = (V, E)$ is called {\em $(s, s')$-dismantlable} if its vertices can be ordered $v_1, \dots, v_n$ so that for each vertex $v_k$, $1 \le k < n$, there exists another vertex $v_{\ell}$ with
$\ell > k$, such that $D_{G-v_{\ell}}(v_k,s) \cap V_k \subseteq D_{G}(v_{\ell},s')$, where $V_k := \{v_k, v_{k+1}, \dots , v_n\}$ and $D_{G-v_{\ell}}(v_k,s)$ is the disk of radius $s$ and with center $v_{k}$ considered in the graph $G\setminus\{v_{\ell}\}$. Such a vertex  ordering is called an {\em $(s, s')$-dismantling ordering}. It was proven in  \cite{ChChNiVa11} that $G$ belongs to the class ${\cal CWFR}(s, s')$, $s'\le s$, if and only if $G$ is $(s, s')$-dismantlable\footnote{Note that if $s'> s$, then the cop can always capture the robber by strictly
decreasing at each move his distance to the robber.}. 
 Furthermore, any $\delta$-hyperbolic graph 
belongs to the class
${\cal CWFR}(2r, r + 2\delta)$ for any $r > 0$, and the graphs in ${\cal CWFR}(s, s')$ are
$(s - 1)$-hyperbolic for any $s \ge 2s'$~\cite{ChChNiVa11}. In the follow-up paper~\cite{ChChPaPe14}, those results were complemented by proving that if $s' < s$, then any graph of
${\cal CWFR}(s, s')$ is $\delta$-hyperbolic with  $\delta = O(s^2)$ (unfortunately, a large constant is hidden under the big-O). Authors of \cite{ChChPaPe14} showed  also that the dependency between $\delta$  and $s$ is linear if $s- s' = \Omega(s)$ and $G$
satisfies a slightly stronger $(s, s')^*$-dismantling condition. A graph $G = (V, E)$ is called {\em $(s, s')^*$-dismantlable} if its vertices can be ordered $v_1, \dots, v_n$ so that for each vertex $v_k$, $1 \le k < n$, there exists another vertex $v_{\ell}$ with
$\ell > k$, such that $D_G(v_{k},s) \cap V_k \subseteq D_G(v_{\ell},s')$, where $V_k := \{v_k, v_{k+1}, \dots , v_n\}$ and both disks are considered in $G$. Such a vertex  ordering is called an {\em $(s, s')^*$-dismantling ordering}. Paper~\cite{ChChPaPe14} demonstrated also that weakly modular graphs from ${\cal CWFR}(s, s')$ with $s' < s$ are $184s$-hyperbolic. 

Next, we show that $\alpha_i$-metric graphs with $n>1$ vertices belong to the class
${\cal CWFR}(r, \lceil{r/2}\rceil + 2i+1)$ for any $r>0$. In fact, we will prove that they are $(r, \lceil{r/2}\rceil + 2i+1)^*$-dismantlable. For that, we will need Lemma \ref{lem:thinness}. 

\begin{lemma}\label{(s,s')-dism}
If $G=(V,E)$ is an $\alpha_i$-metric graph with $n>1$ vertices, then every BFS-ordering of $G$ is also an  $(r, \lceil{r/2}\rceil + 2i+1)^*$-dismantling ordering for every $r>0$.  
\end{lemma}

\begin{proof} We may assume that $r\ge 2$. 
Let $v_1, \dots, v_n$ with $v_n:=u$ be a $BFS(u)$-ordering of $G$ started at an arbitrary vertex $u$. We will show that for every vertex $v_k$ there is a vertex $v_{\ell}$ with $\ell>k$ such that for every vertex $x\in \{v_{k+1}, \dots, v_n\}$, $d(x,v_k)\le r$ implies $d(x,v_{\ell})\le \lceil{r/2}\rceil+2i+1$. As a vertex $v_{\ell}$  we will choose any vertex from  $S_q(v_k,v_n)$, where $q=\min\left\{\lfloor{r/2}\rfloor,d(v_k,v_n)\right\}$ (in particular, if $d(v_k,v_n)<\lfloor{r/2}\rfloor$, then we set $v_{\ell}:=v_n$). 
Let $x$ be an arbitrary vertex from $\{v_{k+1}, \dots, v_n\}$ with $d(x,v_k)\le r$, and $x'$ be a vertex of $S_q(v_k,v_n)$ closest to $x$.
If $x=x'$, then by Lemma~\ref{lem:thinness}, we get $d(x,v_{\ell}) \le i+1 \le \lceil{r/2}\rceil +2i+1$.
Thus, from now on $x \ne x'$, and let $x''$ be a neighbour of $x'$ on a shortest path from $x'$ to $x$. By the choice of $x'$, $d(x'',v_k)>d(x',v_k)$ or $d(x'',v_n)>d(x',v_n)$. In the former case, by the $\alpha_i$-metric property applied to $v_k,x',x'', x$, we get $d(v_k,x)\ge d(v_k,x')+d(x',x)-i$. That is, $d(x',x)\le d(v_k,x)-d(v_k,x')+i$. If $q < \lfloor{r/2}\rfloor$ then necessarily $x' = v_n$, and so $d(v_k,x)-d(v_k,x')+i = d(v_k,x) - d(v_k,v_n) + i \le d(v_n,x)+i$. We also know, by the property of a $BFS(u)$-ordering, that $d(v_n,v_k)\ge d(v_n,x)$. Therefore, $d(x',x) \le d(v_n,x)+i \le d(v_n,v_k)+i < \lfloor{r/2}\rfloor + i$. Otherwise ($q = \lfloor{r/2}\rfloor$), $d(v_k,x)-d(v_k,x')+i \le r- \lfloor{r/2}\rfloor+i=\lceil{r/2}\rceil +i$. Finally, in the latter case, by the $\alpha_i$-metric property applied to $v_n,x',x'', x$, we get $d(v_n,x)\ge d(v_n,x')+d(x',x)-i$. Recall that we also know, by the property of a $BFS(u)$-ordering, that $d(v_n,v_k)\ge d(v_n,x)$. Hence, $d(x',x)\le d(v_n,x)-d(v_n,x')+i\le d(v_n,v_k)-d(v_n,x')+i\le d(x',v_k)+i\le \lfloor{r/2}\rfloor+i$. 

Since in either case $d(x',x)\le \lceil{r/2}\rceil +i$ holds and, by Lemma \ref{lem:thinness}, $d(x',v_{\ell})\le i+1$, we get $d(x,v_{\ell})\le \lceil{r/2}\rceil +2i+1$.
\qed
\end{proof}

From Lemma \ref{(s,s')-dism} and \cite[Corollary 3]{ChChPaPe14}, we conclude: 
\begin{corollary} \label{cor:hyp-huge}
 Every $\alpha_i$-metric graph is $\delta$-hyperbolic for 
 $\delta=O(i)$. 
\end{corollary}

\begin{proof} By Lemma \ref{(s,s')-dism}, any  $\alpha_i$-metric graph $G$ with $n>1$ vertices has an  $(r, \lceil{r/2}\rceil + 2i+1)^*$-dismantling ordering for every $r>0$. It is known  \cite[Corollary 3]{ChChPaPe14} that if a graph $G$ is 
$(s, s')^*$-dismantlable with $s - s' \ge k s$ for some constant $k > 0$, then $G$ is $64s/k$-hyperbolic. 
For our  $\alpha_i$-metric graph $G$, we can pick $k:=1/4$ and $r:= 4(2i+1)$. Then,  $G$ is  $(s, s')^*$-dismantlable with $s=4(2i+1)$, $s'=3(2i+1)$, and $s=s'+(1/4)s$. Therefore, $G$ must be $\delta$-hyperbolic for $\delta=2^6s/k=2^{10}(2i+1).$   
\qed
\end{proof}


    \subsection{Using Geodesic Triangles}\label{sec:triangles}

For every vertices $x,y,z$ in a graph $G=(V,E)$, the \emph{Gromov product} of $x,y$ with respect to $z$ is defined as $(x \mid y)_z = \frac 1 2 \left(d(x,z)+d(y,z) - d(x,y)\right)$.
Let $\Delta(x,y,z) = P(x,y) \cup P(y,z) \cup P(z,x)$ be some geodesic triangle.
There is a canonical distance-preserving embedding of $x,y,z$ in some weighted star $T(x,y,z)$: the leaves are $x,y,z$, and the respective weights of their incident edges are $\alpha_x = (y \mid z)_x, \alpha_y = (z \mid x)_y, \alpha_z = (x \mid y)_z$.
Furthermore, if we replace each edge of $T(x,y,z)$ by some continuous segment, then one can define a unique mapping $\varphi : \Delta(x,y,z) \mapsto T(x,y,z)$ such that the restriction of $\varphi$ on either side $P(x,y)$, $P(y,z)$ or $P(z,x)$ is an isometry.
The \emph{thinness} of  $\Delta(x,y,z)$ equals the smallest $\delta$ such that, for every $u,v \in \Delta(x,y,z)$ with $\varphi(u) = \varphi(v)$, $d_G(u,v) \le \delta$. 
See Fig.~\ref{fig:thinness} taken from~\cite{Fast-appr-hyperb} for an illustration.

\begin{figure}
    \centering
    \includegraphics[width=\textwidth]{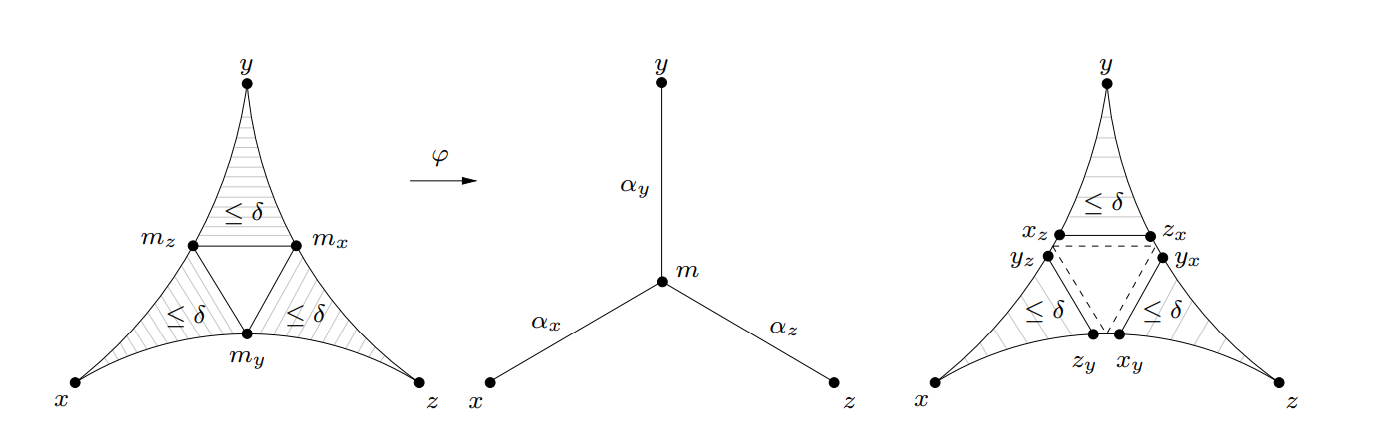}
    \caption{Thinness in graphs.}
    \label{fig:thinness}
    \vspace*{-5mm}
\end{figure}

It is known \cite[Proposition 3.1]{Fast-appr-hyperb} that in $\delta$-hyperbolic graphs all geodesic triangles are $4\delta$-thin. 
Conversely \cite[Lemma 3.1]{ChDrVa17}, if all geodesic triangles in a graph are $\delta$-thin, then it is $\delta$-hyperbolic.
Next we show that the thinness of geodesic triangles in $\alpha_i$-metric graphs, and so, their hyperbolicity, is at most $3(i+1)$. 

\begin{theorem}\label{thm:thinness}
All geodesic triangles in an $\alpha_i$-metric graph $G$ are $3(i + 1)$-thin. 
\end{theorem}

\begin{proof} Consider an arbitrary triple $x,y,z$ of vertices. By \cite[Lemma 3.2]{Fast-appr-hyperb} and by symmetry, it is sufficient to show that any two vertices $z'\in S_k(y,x)$ and $x'\in S_k(y,z)$, where $k=\lfloor(x|z)_y\rfloor$, are at distance at most $3(i+1)$. To do so, we will show that a vertex $z^*$ exists in $S_k(y,x)$ whose distance to any vertex $x'\in S_k(y,z)$ is at most $2(i+1)$. Then, by Lemma~\ref{lem:thinness}, it will follow that $d(z',x')\le 3(i+1)$ for any two vertices $z'\in S_k(y,x)$ and $x'\in S_k(y,z)$.  

Consider three disks $D(y,r_y), D(x,r_x), D(z,r_z)$, where  $r_y=k=\lfloor(x|z)_y\rfloor$, $r_x=\lceil(z|y)_x\rceil$ and $r_z=\lceil(x|y)_z\rceil$. Since  $\lfloor(x|z)_y\rfloor+\lceil(z|y)_x\rceil=d(x,y)$, $\lfloor(x|z)_y\rfloor+\lceil(x|y)_z\rceil=d(y,z)$ and $\lceil(z|y)_x\rceil+\lceil(x|y)_z\rceil\ge d(x,z)$, those disks pairwise intersect. By Lemma~\ref{lem:three-balls}, there is a vertex $z^*\in S_k(y,x)$ whose distance to $z$ is at most $r_z+i=\lceil(x|y)_z\rceil+i$. 
Applying now Lemma~\ref{lm:auxiliary-GD} to $z,y,z^*$ and any vertex $x'\in S_k(y,z)$, we get $d(z^*,x')\le i+i+2=2i+2$.  
\qed
\end{proof}

It was further proved in~\cite[Proposition 10]{ChDEHV08} that if the interval thinness of a graph $G$ is at most $\mu$, and the maximum side-length of its metric triangles is at most $\nu$, then $G$ must be $(16\mu+4\nu)$-hyperbolic. By Lemma~\ref{lem:thinness}, $\mu \le i+1$ if $G$ is $\alpha_i$-metric. Furthermore, if $i=1$, then it was proved in~\cite{BaCh2003} that $\nu \le 2$ (see also Lemma~\ref{lem:triangle} in the next section). However, we conclude   this section by proving that for every $i \ge 2$, the maximum side-length of metric triangles in $\alpha_i$-metric graphs is unbounded. 

\begin{lemma}\label{lem:unbounded-metric-triangle}
    For every integer $p \ge 2$, there exists an $\alpha_2$-metric graph $G_p$ with $4p$ vertices and a metric triangle $x_0y_{p-1}z_{p-1}$ such that $d(x_0,y_{p-1}) = d(x_0,z_{p-1}) = p \ge 2 = d(y_{p-1},z_{p-1})$.
    In particular, there is no constant upper bound on the side-length of metric triangles in an $\alpha_2$-metric graph. 
\end{lemma}
\begin{proof}
    The proof is by induction on $p$.
    The vertex set of $G_p$ ($p \ge 1$) consists of four disjoint sets: $W = \{w_0,w_1,\ldots,w_{p-1}\}$; $X = \{x_0,x_1,\ldots,x_{p-1}\}$; $Y = \{y_0,y_1,\ldots,y_{p-1}\}$; and $Z = \{z_0,z_1,\ldots,z_{p-1}\}$.
    Its edge set is as follows:
    \begin{itemize}
        \item for every $i$ such that $0 \leq i < p-1$, $y_i \sim y_{i+1}$ and $z_i \sim z_{i+1}$;
        \item for every $i$ such that $0 \le i < p$, $x_i,y_i,w_i,z_i$ induce a $C_4$;
        \item for every $i$ such that $0 \le i < p -1$, $w_i \sim y_{i+1},z_{i+1}$;
        \item for every $i$ such that $0 < i \le p-1$, $x_i \sim y_{i-1},z_{i-1}$.
    \end{itemize}


    \begin{figure}[!h]
        \centering
        \resizebox{2.5in}{!}{
        \begin{tikzpicture}
            \node[circle,fill=black,inner sep=0pt,minimum size=5pt, label=above:{$x_0$}] at (0,4) {};
            \node[circle,fill=black,inner sep=0pt,minimum size=5pt, label=below left:{$y_0$}] at (-4,0) {};
            \node[circle,fill=black,inner sep=0pt,minimum size=5pt, label=below right:{$z_0$}] at (4,0) {};
            \node[circle,fill=black,inner sep=0pt,minimum size=5pt, label=below:{$w_0$}] at (0,-4) {};
            \node[circle,fill=black,inner sep=0pt,minimum size=5pt, label=above:{$w_1$}] at (0,3) {};
            \node[circle,fill=black,inner sep=0pt,minimum size=5pt, label=above left:{$y_1$}] at (-3,0) {};
            \node[circle,fill=black,inner sep=0pt,minimum size=5pt, label=above right:{$z_1$}] at (3,0) {};
            \node[circle,fill=black,inner sep=0pt,minimum size=5pt, label=below:{$x_1$}] at (0,-3) {};
            \node[circle,fill=black,inner sep=0pt,minimum size=5pt, label=above:{$x_2$}] at (0,2) {};
            \node[circle,fill=black,inner sep=0pt,minimum size=5pt, label=below left:{$y_2$}] at (-2,0) {};
            \node[circle,fill=black,inner sep=0pt,minimum size=5pt, label=below right:{$z_2$}] at (2,0) {};
            \node[circle,fill=black,inner sep=0pt,minimum size=5pt, label=below:{$w_2$}] at (0,-2) {};
            \node[circle,fill=black,inner sep=0pt,minimum size=5pt, label=above:{$w_3$}] at (0,1) {};
            \node[circle,fill=black,inner sep=0pt,minimum size=5pt, label=above left:{$y_3$}] at (-1,0) {};
            \node[circle,fill=black,inner sep=0pt,minimum size=5pt, label=above right:{$z_3$}] at (1,0) {};
            \node[circle,fill=black,inner sep=0pt,minimum size=5pt, label=below:{$x_3$}] at (0,-1) {};
            \draw[thick] (0,4) -- (-4,0) -- (0,-4) -- (4,0) -- (0,4); \draw[thick] (0,3) -- (-3,0) -- (0,-3) -- (3,0) -- (0,3); \draw[thick] (0,2) -- (-2,0) -- (0,-2) -- (2,0) -- (0,2); \draw[thick] (0,1) -- (-1,0) -- (0,-1) -- (1,0) -- (0,1);
            \draw[thick] (4,0) -- (1,0); \draw[thick] (-1,0) -- (-4,0);
            \draw (4,0) -- (0,-3) -- (-4,0); \draw (3,0) -- (0,-4) -- (-3,0);
            \draw (2,0) -- (0,3) -- (-2,0); \draw (3,0) -- (0,2) -- (-3,0);
            \draw (1,0) -- (0,-2) -- (-1,0); \draw (2,0) -- (0,-1) -- (-2,0);
        \end{tikzpicture}}
        \caption{Graph $G_4$ in the proof of Lemma~\ref{lem:unbounded-metric-triangle}.}
        \label{fig:gp}
        \vspace*{-5mm}
    \end{figure}

    See Fig.~\ref{fig:gp} for an illustration.
    We can compute from the above construction the following distances (by induction on $p$): 
    \begin{enumerate}
        \item\label{dist-formula-gp-1} $d(y_i,y_j) = d(z_i,z_j) = |j-i|$;
        \item\label{dist-formula-gp-2} $d(y_i,z_i) = 2$, and $d(y_i,z_j) = |j-i|+1$ if $i \ne j$;
        \item\label{dist-formula-gp-3} $d(x_i,y_j) = d(x_i,z_j) = 1+j-i$ if $j \ge i$, and $d(x_i,y_j) = d(x_i,z_j) = i-j$ if $j < i$;
        \item\label{dist-formula-gp-4} $d(w_i,y_j) = d(w_i,z_j) = 1+i-j$ if $j \le i$, and $d(w_i,y_j) = d(w_i,z_j) = j-i$ if $j > i$;
        \item\label{dist-formula-gp-5} $d(x_i,x_j) = d(w_i,w_j) = |j-i|+1$ if $i \ne j$;
        \item\label{dist-formula-gp-6} $d(x_i,w_j) = j-i+2$ if $i \le j$, $d(x_i,w_{i-1}) = 2$, and $d(x_i,w_j) = i-j$ if $i > j+1$.
    \end{enumerate}
    In particular, for $p \ge 2$, we can directly deduce from the above distance formulas that $x_0y_{p-1}z_{p-1}$ is a metric triangle because the open intervals $I^o(x_0,y_{p-1})= Y \setminus \{y_{p-1}\}$, $I^o(x_0,z_{p-1})= Z \setminus \{z_{p-1}\}$, $I^o(y_{p-1},z_{p-1})=\{w_{p-2},x_{p-1},w_{p-1}\}$ are pairwise disjoint.
    Furthermore, $d(x_0,y_{p-1}) = d(x_0,z_{p-1}) = p$, and $d(y_{p-1},z_{p-1}) = 2$.

    \smallskip
    It remains to prove that $G_p$ is an $\alpha_2$-metric graph.
    For $p=1$, this is true, because $G_1 = C_4$.
    From now on, we assume $p > 1$ and $G_{p-1}$ is an $\alpha_2$-metric graph.
    Then, $H_1 = G_p \setminus \{w_{p-1},x_{p-1},y_{p-1},z_{p-1}\}$ is an isometric subgraph of $G_p$ that is isomorphic to $G_{p-1}$.
    Furthermore, let us consider the automorphism $\varphi$ of $G_p$ such that, for every $i$ such that $0 \le i < p$, we have $\varphi(w_i) = x_{p-1-i}$, $\varphi(x_i) = w_{p-1-i}$, $\varphi(y_i) = y_{p-1-i}$ and $\varphi(z_i) = z_{p-1-i}$.
    We obtain that $H_2 = G_p \setminus \{w_0,x_0,y_0,z_0\} = \varphi(H_1)$ is also an isometric copy of $G_{p-1}$.
    We next observe that every edge of $G_p$ must be covered by at least one of $H_1,H_2$.
    Hence, due to the aforementioned isomorphism between $H_1$ and $H_2$, it is sufficient to prove the $\alpha_2$-metric property in $G_p$ for every $4$-tuple $s,u,v,t$ such that $uv$ is an edge of $H_1$.
    For that, let us pick arbitrary vertices $s$ and $t$ such that $d(s,u)<d(s,v)$ and $d(t,v)<d(t,u)$.
    By induction, $H_1$ is $\alpha_2$-metric.
    Therefore, we may assume in what follows that $s \notin V(H_1)$.
    Assume furthermore $s \sim u$. Then, $d(s,u)+d(u,t)-2 = d(u,t)-1 \le d(s,t)+1-1 = d(s,t)$, and we are done.
    Hence, from now on, we assume $s \not\sim u$ and, similarly, $t \not\sim v$.
    In particular, $u,v \notin N[s] \cup N[t]$.

    Suppose first $t \notin V(H_1)$.
    If $s=y_{p-1}$, $t=z_{p-1}$, then due to distance formulas (\ref{dist-formula-gp-3}) and (\ref{dist-formula-gp-4}), we obtain $u,v \in Y \cup Z$.
    Furthermore, due to distance formulas (\ref{dist-formula-gp-1}) and (\ref{dist-formula-gp-2}), $u \in Y$ and $v \in Z$.
    However, the latter is impossible because there are no edges between $Y$ and $Z$.
    Therefore, $\{s,t\} \cap \{x_{p-1},w_{p-1}\} \ne \emptyset$.
    By construction of $G_p$, for every $r \notin V(H_1)$, $D(w_{p-1},2) \subseteq D(r,2) \subseteq D(x_{p-1},2)$. 
    In particular, we may assume up to reverting the respective roles of $s$ and $t$ that $d(s,r') \le d(t,r')$ for every $r'\notin N[s] \cup N[t]$.
    As $u,v \notin N[s] \cup N[t]$, $d(s,u) < d(s,v) \le d(t,v) < d(t,u)$.
    Therefore, $d(s,u) \le d(t,u)-2$.
    Furthermore, $d(s,u) \ge d(t,u) - d(s,t) \ge d(t,u)-2$.
    This implies $d(s,u) = d(t,u)-2$, and so, $d(s,t) = 2$.
    As $\{s,t\} \cap \{x_{p-1},w_{p-1}\} \ne \emptyset$, the only possibility is $s = x_{p-1},t=w_{p-1}$.
    However, as $N(x_{p-1}) \subseteq N[y_{p-1}] \cup N[z_{p-1}] \subseteq D(w_{p-1},2)$, $|d(x_{p-1},r')-d(w_{p-1},r')| \le 1$ for every $r' \in V(H_1)$, which holds in particular for $r'=u$, a contradiction.
    As a result, $t \in V(H_1)$.

    \begin{itemize}
        \item {\it Case $s = x_{p-1}$}.
        By construction of $G_p$, $N(x_{p-1}) = N(w_{p-2})$.
        Therefore, for every $r \in V(H_1) \setminus \{w_{p-2}\}$, we have $d(r,x_{p-1}) = d(r,w_{p-2})$.
        As $s \not\sim v$, it implies $u \neq w_{p-2}$.
        In the same way, as $s \not\sim u$, it implies $v \neq w_{p-2}$.
        Under the assumptions $u \neq w_{p-2}$ and $v \neq w_{p-2}$, we obtain $d(w_{p-2},u) = d(s,u) < d(s,v) = d(w_{p-2},v)$. As $d(t,v) < d(t,u)$, it also implies $t \ne w_{p-2}$.
        Then, as by the induction hypothesis, $H_1$ is $\alpha_2$-metric, by the $\alpha_2$-metric property applied to $w_{p-2},u,v,t$ in $H_1$, $d(s,t) = d(w_{p-2},t) \ge d(w_{p-2},u)+d(u,t)-2 = d(s,u)+d(u,t)-2$.
        \item {\it Case $s = w_{p-1}$}.  Due to distance formulas (\ref{dist-formula-gp-4}) and (\ref{dist-formula-gp-5}) applied for $i \in \{p-2,p-1\}, j \le p-2$, and to the distance formula (\ref{dist-formula-gp-6}) applied for $i \le p-2, j \in \{p-2,p-1\}$, for every $r \in V(H_1) \setminus \{w_{p-2}\}$, we have $d(r,w_{p-1}) = d(r,w_{p-2})+1$.
        Furthermore, since $d(w_{p-1},w_{p-2})=2$, we get that $w_{p-1}$ is equidistant to $w_{p-2}$ and any of its neighbours in $H_1$.
        In particular, $w_{p-2} \ne u,v$.
        Then, as $d(w_{p-2},u) = d(s,u)-1<d(s,v)-1 = d(w_{p-2},v)$, and $d(t,v) < d(t,u)$, it implies $t \ne w_{p-2}$. 
        As by the induction hypothesis, $H_1$ is $\alpha_2$-metric, by the $\alpha_2$-metric property applied to $w_{p-2},u,v,t$ in $H_1$, $d(s,t) = d(w_{p-2},t)+1 \ge (d(w_{p-2},u)+d(u,t)-2)+1 =  (d(w_{p-2},u)+1) + d(u,t)-2 = d(s,u)+d(u,t)-2$.
        \item {\it Case $s = y_{p-1}$ or $s = z_{p-1}$}.
        By symmetry, we only consider the case $s = y_{p-1}$.

        \smallskip
        Let us first consider the subcase $\{u,v\} \cap \{y_{p-2},z_{p-2},w_{p-2}\} \ne \emptyset$.
        Since $s \not\sim u$, the edge $uv$ must be incident to $z_{p-2}$ and another non-neighbour of $s$.
        By construction of $G_p$, the neighbours of $z_{p-2}$ in $H_1$ that are nonadjacent to $y_{p-1}$ are: $x_{p-2}$, and (only if $p \ge 3$) $w_{p-3},z_{p-3}$. 
        Furthermore, as $y_{p-2} \in N(s) \cap N(x_{p-2})$, $d(s,x_{p-2}) = 2 = d(s,z_{p-2})$, and in the same way if $w_{p-3}$ exists then due to the distance formula (\ref{dist-formula-gp-4}), $d(s,w_{p-3}) = 2 = d(s,z_{p-2})$.
        Since we must have $d(s,u) < d(s,v)$, $u = z_{p-2}$ and $v = z_{p-3}$ is the only possibility.
        As $d(t,v) < d(t,u)$, and $d(t,v) \ge 2$, it implies that $t \notin \{w_{p-2},x_{p-2},y_{p-2},z_{p-2}\}$. Furthermore, to prove that $d(s,u)+d(u,t)-2 \le d(s,t)$, as $d(s,u) = 2$, it is necessary and sufficient to prove $d(u,t)\le d(s,t)$. For that, assume $t = y_j$, for some $j \le p-3$.
        Due to the distance formula (\ref{dist-formula-gp-1}), $d(s,t) = p-1-j$.
        If $j = p-3$ then due to the distance formula (\ref{dist-formula-gp-2}) $d(t,u) = d(t,v)$, a contradiction.
        Therefore, $j \le p-4$, and again due to the distance formula (\ref{dist-formula-gp-2}) $d(t,u) = d(t,v)+1 = p-1-j = d(s,t)$. 
        Assume now $t = z_j$, for some $j \le p-3$. 
        Due to the distance formula (\ref{dist-formula-gp-2}), $d(s,t) = p-j$, and due to the distance formula (\ref{dist-formula-gp-1}), $d(t,u) = p-2-j < d(s,t)$.
        Otherwise, $t \in \{w_j,x_j\}$, for some $j \le p-3$.
        Due to distance formulas (\ref{dist-formula-gp-3}) and (\ref{dist-formula-gp-4}), $d(s,t) = d(u,t)+1$.

        \smallskip
        From now on, $\{u,v\} \cap \{y_{p-2},z_{p-2},w_{p-2}\} = \emptyset$.
        We claim that for every $r \in V(H_1) \setminus \{y_{p-2},z_{p-2},w_{p-2}\}$, $d(r,y_{p-1}) = d(r,y_{p-2})+1$.
        Indeed, if $r \in Y$ then it follows from the distance formula (\ref{dist-formula-gp-1}) applied for $i \in \{p-2,p-1\}$, $j \le p-3$.
        In the same way if $r \in Z$, then it follows from the distance formula (\ref{dist-formula-gp-2}) applied for $i \in \{p-2,p-1\}$, $j \le p-3$.
        If $r \in W$, then the claim follows from the distance formula (\ref{dist-formula-gp-4}) applied for $j > i$, $j \in \{p-2,p-1\}$.
        Finally, if $r \in X$, then it follows from the distance formula (\ref{dist-formula-gp-3}) applied for $j \ge i$, $j \in \{p-2,p-1\}$.
        We will prove next that $t \notin \{y_{p-2},z_{p-2},w_{p-2}\}$.
        Under this assumption, as by the induction hypothesis $H_1$ is $\alpha_2$-metric, by the $\alpha_2$-metric property applied to $y_{p-2},u,v,t$ in $H_1$, we obtain $d(t,y_{p-1}) = d(t,y_{p-2})+1 \ge (d(t,u)+d(u,y_{p-2})-2)+1 = d(t,u)+d(u,y_{p-1})-2$.
        So, suppose by contradiction $t \in \{y_{p-2},z_{p-2},w_{p-2}\}$.
        As $d(y_{p-2},u) = d(s,u)-1 < d(s,v)-1 = d(y_{p-2},v)$, and $d(t,v) < d(t,u)$, we obtain $y_{p-2} \neq t$.
        Therefore, either $t = z_{p-2}$, or $t = w_{p-2}$.
        As $N_{H_1}(w_{p-2}) = \{y_{p-2},z_{p-2}\}$, and $H_1$ is an isometric subgraph of $G$, necessarily $z_{p-2} \in I(t,v)$.
        In particular, $d(z_{p-2},v) < d(z_{p-2},u)$.
        However, due to distance formulas (\ref{dist-formula-gp-1})--(\ref{dist-formula-gp-4}), it implies that $u \in Y$, $v \in Z$, which is impossible.
    \end{itemize}
    Overall, the case analysis implies that any $s,u,v,t$ satisfies the $\alpha_2$-metric property in $G_p$. \qed
\end{proof}


    \subsection{Using Injective Hulls}\label{sec:helly}

A graph is called \emph{Helly} if every family of pairwise intersecting disks has a nonempty common intersection.
It is known ({\it e.g.}, see~\cite{ChDrVa17}) that every graph $G$ can be isometrically embedded in a (unique) smallest Helly graph ${\cal H}(G)$, which is sometimes called the \emph{injective hull} of $G$.
In what follows, we will show that using some recent results on injective hulls of graphs, a better bound on the hyperbolicity of $\alpha_i$-metric graphs can be obtained. First note that, unfortunately, the injective hull of an $\alpha_i$-metric graph may not necessarily be $\alpha_i$-metric. This complicates some considerations but the difficulties can be circumvented. In Fig.\ref{fig-thinness-sharp} we give an example of an $\alpha_1$-metric graph whose injective hull is not $\alpha_1$-metric. To have other such examples, one main ingredient is the sharpness of the bound of Lemma~\ref{lem:thinness}. In particular, that means that there are $\alpha_1$-metric graphs $G$ whose injective hulls ${\cal H}(G)$ are not $\alpha_1$-metric, because every $\alpha_1$-metric graph has convex disks~\cite{YuCh1991}, however the interval thinness of a Helly graph with convex disks must be at most $1$~\cite{DuDr21-netw}. 


\begin{figure}[htb]
    \centering
    \begin{subfigure}{.45\textwidth}
        \centering
        \begin{tikzpicture}
            \node[circle,fill=black,inner sep=0pt,minimum size=5pt] at (-2,1) {};
            \node[circle,fill=black,inner sep=0pt,minimum size=5pt] at (-1,.5) {};
            \node[circle,fill=black,inner sep=0pt,minimum size=5pt] at (-1,1.5) {};
            \node[circle,fill=black,inner sep=0pt,minimum size=5pt] at (0,2) {};
            \node[circle,fill=black,inner sep=0pt,minimum size=5pt] at (0,1) {};
            \node[circle,fill=black,inner sep=0pt,minimum size=5pt] at (0,0) {};
            \node[circle,fill=black,inner sep=0pt,minimum size=5pt] at (2,1) {};
            \node[circle,fill=black,inner sep=0pt,minimum size=5pt] at (1,.5) {};
            \node[circle,fill=black,inner sep=0pt,minimum size=5pt] at (1,1.5) {};
            \draw (-2,1) -- (0,2) -- (2,1) -- (0,0) -- (-2,1);
            \draw (-1,.5) -- (-1,1.5) -- (1,.5) -- (1,1.5) -- (-1,.5);
            \draw (0,2) -- (0,0);
            \draw node at (0,-.25) {};
        \end{tikzpicture}
        \caption{An $\alpha_1$-metric graph $G$ with interval thinness 2.}
        \label{fig:thinness-sharp-a}
    \end{subfigure}\hfill
    \begin{subfigure}{.45\textwidth}
        \centering
        \begin{tikzpicture}
            \node[circle,fill=black,inner sep=0pt,minimum size=5pt] at (-2,0) {};
            \node[circle,fill=black,inner sep=0pt,minimum size=5pt] at (-1,-.5) {};
            \node[circle,fill=black,inner sep=0pt,minimum size=5pt] at (-1,.5) {};
            \node[circle,fill=black,inner sep=0pt,minimum size=5pt] at (0,1) {};
            \node[circle,fill=black,inner sep=0pt,minimum size=5pt] at (0,0) {};
            \node[circle,fill=black,inner sep=0pt,minimum size=5pt] at (0,-1) {};
            \node[circle,fill=black,inner sep=0pt,minimum size=5pt] at (2,0) {};
            \node[circle,fill=black,inner sep=0pt,minimum size=5pt] at (1,-.5) {};
            \node[circle,fill=black,inner sep=0pt,minimum size=5pt] at (1,.5) {};
            \draw (-2,0) -- (0,1) -- (2,0) -- (0,-1) -- (-2,0);
            \draw (-1,-.5) -- (-1,.5) -- (1,-.5) -- (1,.5) -- (-1,-.5);
            \draw (0,1) -- (0,-1);
            \draw[thick] (-2,0) -- (0,0); \draw[thick] (-1,-.5) -- (-.65,0) -- (-1,.5);
            \draw[thick] (2,0) -- (0,0); \draw[thick] (1,-.5) -- (.65,0) -- (1,.5);
            \draw[very thick] (-.65,0) -- (0,1) -- (.65,0) -- (0,-1) -- (-.65,0);
            \node[rectangle,fill=red,inner sep=0pt,minimum size=5pt] at (-.65,0) {};
            \node[rectangle,fill=red,inner sep=0pt,minimum size=5pt] at (.65,0) {};
        \end{tikzpicture}
        \caption{The injective hull of $G$, which is not $\alpha_1$-metric (it contains an induced $C_4$).}
        \label{fig:thinness-sharp-b}
    \end{subfigure}
    \caption{An example showing that injective hulls of $\alpha_1$-metric graphs are not always $\alpha_1$-metric.}
    \label{fig-thinness-sharp}
    \vspace*{-5mm}
\end{figure}

\medskip
We recall the following properties of injective hulls. 

\begin{lemma}[see Corollary 2 in~\cite{GDL22}]\label{lem:hyp}
A graph $G$ is $\delta$-hyperbolic if and only if its injective hull ${\cal H}(G)$ is $\delta$-hyperbolic.
\end{lemma}

\begin{lemma}[see Theorem 2 in~\cite{DrGu19}]\label{lem:helly}
If $H$ is Helly and its interval thinness is at most $\tau$, then it is $\left\lceil\frac{\tau}{2}\right\rceil$-hyperbolic.
\end{lemma}

\begin{lemma}[see Proposition 2 in~\cite{DrGu21}]\label{lem:max-sp}
If $H$ is the injective hull of $G$ and $x,y \in V(H)$ are arbitrary, then there exist $x',y' \in V(G)$ such that any shortest $(x,y)$-path in $H$ is contained in a shortest $(x',y')$-path in $H$.
\end{lemma}

We complement these above properties with the following simple observation about distances in the injective hull of a graph:
\begin{lemma}\label{lem:dist-in-hull}
    If $H$ is the injective hull of $G$ and $x,y$ are vertices of $H$ such that $d(x,v) \le d(y,v) + \lambda$ for every $v \in V(G)$, then $d(x,y) \le \lambda$.
\end{lemma}
\begin{proof}
    By Lemma~\ref{lem:max-sp}, there exist $u,v \in V(G)$ such that any shortest $(x,y)$-path in $H$ is contained in a shortest $(u,v)$-path in $H$.
    Without loss of generality, $d(v,y) \le d(v,x)$.
    Since $d(v,x) \le d(v,y) + \lambda$, $d(x,y) = d(v,x) - d(v,y) \le \lambda$.
\end{proof}

We are now ready to prove our main result in this section.

\begin{theorem}\label{thm:main}
If $G=(V,E)$ is $\alpha_i$-metric, then it is $\delta$-hyperbolic for some $\delta \leq i + \left\lceil\frac{i+1}{2}\right\rceil \leq \frac 3 2 \cdot (i+1)$.
\end{theorem}

\begin{proof}
Set $\delta = i + \left\lceil\frac{i+1}{2}\right\rceil$.
By Lemma~\ref{lem:hyp}, $G$ and its injective hull ${\cal H}(G)$ have the same hyperbolicity constant. 
Therefore, we are left proving that ${\cal H}(G)$ is $\delta$-hyperbolic. 
For that, by Lemma~\ref{lem:helly}, it is sufficient to prove that the interval thinness of ${\cal H}(G)$ is at most $2\delta$. 
Furthermore, in order to bound the interval thinness of ${\cal H}(G)$, by Lemma~\ref{lem:max-sp}, it is sufficient to consider the shortest-paths between vertices $u,v \in V$ ({\it i.e.}, we discard the pairs of vertices with at least one vertex in ${\cal H}(G) \setminus V$). 

Thus, from now on, let $u,v \in V$ be fixed, and let $S_k(u,v,{\cal H}(G))$ be some fixed slice of $I_{{\cal H}(G)}(u,v)$.
By Lemma~\ref{lem:thinness}, the disks $D_G(x,\left\lceil\frac{i+1}{2}\right\rceil)$, for every $x \in S_k(u,v,G)$, pairwise intersect.
Therefore, by the Helly property, there exists a vertex $c$ of ${\cal H}(G)$ such that $S_k(u,v,G) \subseteq D_{{\cal H}(G)}(c,\left\lceil\frac{i+1}{2}\right\rceil)$. Then, in order to prove that every two vertices of $S_k(u,v,{\cal H}(G))$ are pairwise at distance at most $2\delta$, it suffices to prove that $S_k(u,v,{\cal H}(G)) \subseteq D_{{\cal H}(G)}(c,\delta)$.
For that, let $y \in S_k(u,v,{\cal H}(G)) \setminus V$ be arbitrary.
Let us define $r_y(w) = d(y,w)$ for every $w \in V \setminus \{u,v\}$. 
For each $w \in V \setminus \{u,v\}$, by Lemma~\ref{lem:three-balls}, there exists some $x_w \in S_k(u,v,G)$ such that $d(x_w,w) \leq r_y(w) + i$.
It implies $d(c,w) \leq \left\lceil \frac{i+1}{2}\right\rceil + r_y(w) + i = r_y(w) + \delta$.
By Lemma~\ref{lem:dist-in-hull}, $d(y,c) \le \delta$.\qed
\end{proof}

\subsection{The $\alpha_i$-metric property for $1$-hyperbolic graphs}\label{sec:obstructions-alphai}

Note that, although the hyperbolicity of an $\alpha_i$-metric graph is upper-bounded by $i + \left\lceil\frac{i+1}{2}\right\rceil$, there are $n$-vertex 1-hyperbolic graphs that only satisfy the $\alpha_i$-metric property for some $i=\Omega(n)$. Consider, for example, a ladder with height $\ell$ (see Fig. \ref{fig:ladder-c4}). It is a 1-hyperbolic graph and it satisfies the $\alpha_i$-metric property only for $i\ge 2\ell$.  Other examples are obtained by subdiving once some edges of the ladder, see Fig.~\ref{fig:ladder-c5} and~\ref{fig:ladder-c6}.


  \begin{figure}[htb]
      \centering
      \begin{subfigure}{.4\textwidth}
          \begin{tikzpicture}
              \node[circle,fill=black,inner sep=0pt,minimum size=5pt, label=left:{$w$}] at (-3,.5) {};
              \node[circle,fill=black,inner sep=0pt,minimum size=5pt, label=left:{$v$}] at (-3,-.5) {};
              \node[circle,fill=black,inner sep=0pt,minimum size=5pt] at (-2,.5) {};
              \node[circle,fill=black,inner sep=0pt,minimum size=5pt] at (-2,-.5) {};
              \node[circle,fill=black,inner sep=0pt,minimum size=5pt] at (-1,.5) {};
              \node[circle,fill=black,inner sep=0pt,minimum size=5pt] at (-1,-.5) {};
              \node[circle,fill=black,inner sep=0pt,minimum size=5pt] at (0,.5) {};
              \node[circle,fill=black,inner sep=0pt,minimum size=5pt] at (0,-.5) {};
              \node at (1,.5) {$\ldots$};
              \node at (1,-.5) {$\ldots$};
              \node[circle,fill=black,inner sep=0pt,minimum size=5pt] at (2,.5) {};
              \node[circle,fill=black,inner sep=0pt,minimum size=5pt] at (2,-.5) {};
              \node[circle,fill=black,inner sep=0pt,minimum size=5pt, label=right:{$x$}] at (3,.5) {};
              \node[circle,fill=black,inner sep=0pt,minimum size=5pt, label=right:{$u$}] at (3,-.5) {};
              \draw (0,-.5) -- (-3,-.5) -- (-3,.5) -- (0,.5) -- (0,-.5); \draw (-2,.5) -- (-2,-.5); \draw (-1,.5) -- (-1,-.5);
              \draw (2,-.5) -- (2,.5) -- (3,.5) -- (3,-.5) -- (2,-.5);
          \end{tikzpicture}
          \caption{}
          \label{fig:ladder-c4}
      \end{subfigure}\hfill
      \begin{subfigure}{.4\textwidth}
          \begin{tikzpicture}
              \node[circle,fill=black,inner sep=0pt,minimum size=5pt, label=left:{$w$}] at (-3,.5) {};
              \node[circle,fill=black,inner sep=0pt,minimum size=5pt, label=left:{$v$}] at (-3,-.5) {};
              \node[circle,fill=black,inner sep=0pt,minimum size=5pt] at (-2.5,.5) {};
              \node[circle,fill=black,inner sep=0pt,minimum size=5pt] at (-2,.5) {};
              \node[circle,fill=black,inner sep=0pt,minimum size=5pt] at (-2,-.5) {};
              \node[circle,fill=black,inner sep=0pt,minimum size=5pt] at (-1.5,-.5) {};
              \node[circle,fill=black,inner sep=0pt,minimum size=5pt] at (-1,.5) {};
              \node[circle,fill=black,inner sep=0pt,minimum size=5pt] at (-1,-.5) {};
              \node[circle,fill=black,inner sep=0pt,minimum size=5pt] at (0,.5) {};
              \node[circle,fill=black,inner sep=0pt,minimum size=5pt] at (0,-.5) {};
              \node[circle,fill=black,inner sep=0pt,minimum size=5pt] at (-.5,.5) {};
              \node at (1,.5) {$\ldots$};
              \node at (1,-.5) {$\ldots$};
              \node[circle,fill=black,inner sep=0pt,minimum size=5pt] at (2,.5) {};
              \node[circle,fill=black,inner sep=0pt,minimum size=5pt] at (2,-.5) {};
              \node[circle,fill=black,inner sep=0pt,minimum size=5pt, label=right:{$x$}] at (3,.5) {};
              \node[circle,fill=black,inner sep=0pt,minimum size=5pt, label=right:{$u$}] at (3,-.5) {};
              \node[circle,fill=black,inner sep=0pt,minimum size=5pt] at (2.5,-.5) {};
              \draw (0,-.5) -- (-3,-.5) -- (-3,.5) -- (0,.5) -- (0,-.5); \draw (-2,.5) -- (-2,-.5); \draw (-1,.5) -- (-1,-.5);
              \draw (2,-.5) -- (2,.5) -- (3,.5) -- (3,-.5) -- (2,-.5);
          \end{tikzpicture}
          \caption{}
          \label{fig:ladder-c5}
      \end{subfigure}\vfill
      \begin{subfigure}{.4\textwidth}
          \begin{tikzpicture}
              \node[circle,fill=black,inner sep=0pt,minimum size=5pt, label=left:{$w$}] at (-3,.5) {};
              \node[circle,fill=black,inner sep=0pt,minimum size=5pt, label=left:{$v$}] at (-3,-.5) {};
              \node[circle,fill=black,inner sep=0pt,minimum size=5pt] at (-2.5,.5) {};
              \node[circle,fill=black,inner sep=0pt,minimum size=5pt] at (-2.5,-.5) {};
              \node[circle,fill=black,inner sep=0pt,minimum size=5pt] at (-2,.5) {};
              \node[circle,fill=black,inner sep=0pt,minimum size=5pt] at (-2,-.5) {};
              \node[circle,fill=black,inner sep=0pt,minimum size=5pt] at (-1.5,.5) {};
              \node[circle,fill=black,inner sep=0pt,minimum size=5pt] at (-1.5,-.5) {};
              \node[circle,fill=black,inner sep=0pt,minimum size=5pt] at (-1,.5) {};
              \node[circle,fill=black,inner sep=0pt,minimum size=5pt] at (-1,-.5) {};
              \node[circle,fill=black,inner sep=0pt,minimum size=5pt] at (-.5,.5) {};
              \node[circle,fill=black,inner sep=0pt,minimum size=5pt] at (-.5,-.5) {};
              \node[circle,fill=black,inner sep=0pt,minimum size=5pt] at (0,.5) {};
              \node[circle,fill=black,inner sep=0pt,minimum size=5pt] at (0,-.5) {};
              \node at (1,.5) {$\ldots$};
              \node at (1,-.5) {$\ldots$};
              \node[circle,fill=black,inner sep=0pt,minimum size=5pt] at (2,.5) {};
              \node[circle,fill=black,inner sep=0pt,minimum size=5pt] at (2,-.5) {};
              \node[circle,fill=black,inner sep=0pt,minimum size=5pt, label=right:{$x$}] at (3,.5) {};
              \node[circle,fill=black,inner sep=0pt,minimum size=5pt, label=right:{$u$}] at (3,-.5) {};
              \node[circle,fill=black,inner sep=0pt,minimum size=5pt] at (2.5,.5) {};
              \node[circle,fill=black,inner sep=0pt,minimum size=5pt] at (2.5,-.5) {};
              \draw (0,-.5) -- (-3,-.5) -- (-3,.5) -- (0,.5) -- (0,-.5); \draw (-2,.5) -- (-2,-.5); \draw (-1,.5) -- (-1,-.5);
              \draw (2,-.5) -- (2,.5) -- (3,.5) -- (3,-.5) -- (2,-.5);
          \end{tikzpicture}
          \caption{}
          \label{fig:ladder-c6}
      \end{subfigure}
      \caption{Examples of $1$-hyperbolic graphs that are not $\alpha_i$-metric for arbitrarily large values $i$. All these graphs are subdivisions of a ladder with height $\ell$ (Fig.~\ref{fig:ladder-c4}). The value chosen for $\ell$ can be arbitrary for the construction of Fig.~\ref{fig:ladder-c6}, but it must be an even number for that of Fig.~\ref{fig:ladder-c5}. As $v\in I(u,w), w\in I(v,x)$ and $1=d(u,x)=d(u,v)+d(v,x)-2\ell$, the graph of Fig.~\ref{fig:ladder-c4} is not $\alpha_{2\ell-1}$-metric. Similarly, the graph of Fig.~\ref{fig:ladder-c5} is not $\alpha_{3\ell-1}$-metric, and the graph of Fig.~\ref{fig:ladder-c6} is not $\alpha_{4\ell-1}$-metric.}
      \label{fig:ladders}
      \vspace*{-5mm}
  \end{figure}

 The multiple obstructions presented in Fig.~\ref{fig:ladders} hint at the difficulty of characterizing the $1$-hyperbolic $\alpha_i$-metric graphs.
 However, in some simpler settings, ladders turn out to be the only such obstructions.
 We prove it next.
 In what follows, a graph is called \emph{$k$-slim} if for every geodesic triangle $\Delta(u,v,w) = P(u,v) \cup P(v,w) \cup P(w,u)$, every vertex $x$ on one side $P(u,v)$ is at distance at most $k$ to some vertex $x' \in P(v,w) \cup P(w,u)$ on one of the other two sides. In particular, chordal graphs are $1$-slim~\cite{slimness}, and every $k$-slim graph is $\left(2k+\frac 1 2\right)$-hyperbolic~\cite{Sotothesis}.

\begin{theorem}\label{thm:1-slim}
    If $G=(V,E)$ is $1$-slim then, for every integer $i \ge 4$, either $G$ is $\alpha_i$-metric or it contains an isometric ladder of height $\left\lceil \frac{i-1} 2 \right\rceil - 1$.
\end{theorem}

\begin{proof}
Assuming $G$ is not $\alpha_i$-metric, there are vertices $u,v,x,y$ that satisfy the following properties: $u \in I(x,v)$, $v \in I(y,u)$, $u \sim v$, and $d(x,y) \le d(x,u)+d(v,y)-i$. In what follows, we construct an isometric ladder of height $\left\lceil \frac{i-1} 2 \right\rceil - 1$ with all its vertices in $I(x,u) \cup I(y,v)$. The reader can follow our construction using Fig.~\ref{fig:isometric-ladder}.

Let $x' \in I(x,u) \cap I(x,y)$ be maximizing $d(x,x') = k$. Then, $d(x',y) = d(x,y) - k \le d(x,u)-k+d(v,y)-i = d(x',u)+d(v,y) - i$. Therefore, up to replacing $x$ with $x'$, we may assume from now on $I^o(x,u) \cap I^o(x,y) = \emptyset$. In the same way, up to replacing $y$ with some $y' \in I(y,v) \cap I(y,x)$ such that $d(y,y')$ is maximized, we may assume in what follows that $I^o(y,v) \cap I^o(y,x) = \emptyset$.

We fix some shortest paths $P(x,y)$, $P(x,u)$, and $P(y,v)$.
Since we ensured above that $I^o(x,u) \cap I^o(x,y) = \emptyset$, $P(x,y)$ and $P(x,u)$ are internally vertex-disjoint.
In the same way, $P(x,y)$ and $P(y,v)$ are internally vertex-disjoint.
Finally, $P(x,u)$ and $P(y,v)$ are vertex-disjoint since  $d(x,u) < d(x,v)$ and $d(y,v) < d(y,u)$.
Since $d(x,y) \ge d(y,v) - d(x,v) = d(x,u) + d(y,v) - (2d(x,u)+1)$, necessarily, $i \le 2d(x,u)+1$. In the same way, $i \le 2d(y,v)+1$.
We set $\ell = \left\lceil \frac{i-1} 2 \right\rceil - 1$.
Note that $\ell \le \min\{d(x,u),d(y,v)\}$, and that $\ell \ge 1$ because $i \ge 4$.

Then, let $x_\ell,x_{\ell-1},\ldots,x_0 = u$ and $y_\ell,y_{\ell-1},\ldots,y_0=v$ be respectively the terminal sub-paths of $P(x,u)$ and $P(y,v)$.
We prove, in what follows, that $\{x_0,x_1,\ldots,x_\ell\} \cup \{y_0,y_1,\ldots,y_\ell\}$ are the vertices of an isometric ladder of $G$.

\begin{figure}[!ht]
    \centering
    \usetikzlibrary{decorations.pathmorphing, patterns,shapes}
    \begin{tikzpicture}

            \node[circle,fill=black,inner sep=0pt,minimum size=5pt, label=above:{$x$}] at (-1,5) {};
            \node[circle,fill=black,inner sep=0pt,minimum size=5pt, label=above:{$y$}] at (1,5) {};
            \node[circle,fill=black,inner sep=0pt,minimum size=5pt, label=left:{$x_\ell$}] at (-1,3) {};
            \node[circle,fill=black,inner sep=0pt,minimum size=5pt, label=right:{$y_\ell$}] at (1,3) {};
            \node[circle,fill=black,inner sep=0pt,minimum size=5pt, label=left:{$x_{\ell-1}$}] at (-1,2.5) {};
            \node[circle,fill=black,inner sep=0pt,minimum size=5pt, label=right:{$y_{\ell-1}$}] at (1,2.5) {};
            \node[circle,fill=black,inner sep=0pt,minimum size=5pt, label=left:{$x_3$}] at (-1,1.5) {};
            \node[circle,fill=black,inner sep=0pt,minimum size=5pt, label=right:{$y_3$}] at (1,1.5) {};
            \node[circle,fill=black,inner sep=0pt,minimum size=5pt, label=left:{$x_2$}] at (-1,1) {};
            \node[circle,fill=black,inner sep=0pt,minimum size=5pt, label=right:{$y_2$}] at (1,1) {};
            \node[circle,fill=black,inner sep=0pt,minimum size=5pt, label=left:{$x_1$}] at (-1,.5) {};
            \node[circle,fill=black,inner sep=0pt,minimum size=5pt, label=right:{$y_1$}] at (1,.5) {};
            \node[circle,fill=black,inner sep=0pt,minimum size=5pt, label=left:{$u$}] at (-1,0) {};
            \node[circle,fill=black,inner sep=0pt,minimum size=5pt, label=right:{$v$}] at (1,0) {};

            \draw[decorate,decoration=zigzag] (-1,5) -- (1,5);
            \draw[decorate,decoration=zigzag] (-1,5) -- (-1,3);
            \draw[decorate,decoration=zigzag] (1,5) -- (1,3);
            \draw (-1,3) -- (-1,2.5) -- (1,2.5) -- (1,3) -- (-1,3);
            \draw[dashed] (-1,2.5) -- (-1,1.5);
            \draw[dashed] (1,2.5) -- (1,1.5);
            \draw (-1,1.5) -- (-1,0) -- (1,0) -- (1,1.5) -- (-1,1.5);
            \draw (-1,.5) -- (1,.5);
            \draw (-1,1) -- (1,1);

    \end{tikzpicture}
    \caption{Construction of an isometric ladder.}
    \label{fig:isometric-ladder}
    \vspace*{-5mm}
\end{figure}

For that, we first prove that $x_j \sim y_j$ for every $0 \le j \le \ell$.
This is done by finite induction on $j$.
For $j=0$ this is true because $x_0 = u$ and $y_0 = v$.
Then, let $j$ be such that $1 \le j \le \ell$, and $x_{j-1} \sim y_{j-1}$.
We consider the shortest sub-paths $P(x,x_{j-1}),P(y,y_{j-1})$ of $P(x,u),P(y,v)$.
Observe that $P'(x,y_{j-1}) = P(x,x_{j-1})+y_{j-1}$ is a shortest $(x,y_{j-1})$-path.
In particular, $P(x,y), \ P'(x,y_{j-1}), P(y,y_{j-1})$ are the sides of a geodesic triangle formed by vertices $x,y,y_{j-1}$.
Since $G$ is $1$-slim, there exists a $z_j \in P(x,y) \cup P(y,y_{j-1})$ such that $x_j \sim z_j$.
Assume  $z_j \in P(x,y)$.
Since $I^o(x,u) \cap I^o(x,y) = \emptyset$, $z_j \notin I(x,u)$.
Therefore, $d(x,z_j) \ge d(x,x_j) = d(x,x_{j-1})-1$.
Recall that $d(x,y) \le d(x,u)+d(v,y)-i = d(x,x_{j-1})+d(y_{j-1},y) + 2(j-1)-i$.
As a result, $d(z_j,y) = d(x,y)-d(x,z_j) \le d(y_{j-1},y)+2j-i-1$.
Since, furthermore, $I^o(x,y) \cap I^o(y,v) = \emptyset$, $z_j \notin I(y,y_{j-1})$.
This implies $d(y,y_{j-1}) \le d(y,z_j)+d(z_j,y_{j-1})-1 \le d(y,z_j)+2 \le d(y_{j-1},y)+2j-i+1$.
However, $2j-i+1 \le 2\ell-i+1 < 0$, which is impossible.
As a result, $z_j \in P(y,y_{j-1})$.
Since $d(x,u) < d(x,v)$, necessarily, $d(z_j,v) \ge d(x_j,u)$.
Furthermore, since $d(y,v) < d(y,u)$, necessarily, $d(z_j,v) \le d(x_j,u)$.
Therefore, the only possibility is $z_j = y_j$.

For every $j$ such that $0 \le j \le \ell$, we actually proved above that $y_j$ is the only possible neighbour of $x_j$ within $P(y,v)$.
As a result, the subgraph induced by $\{x_0,x_1,\ldots,x_\ell\} \cup \{y_0,y_1,\ldots,y_\ell\}$ is a ladder of height $\ell$.
Finally, observe that for every $j,j'$ such that $0 \le j,j' \le \ell$, $d(x_j,y_{j'}) > |j-j'|$ since $d(x_j,u) < d(x_j,v)$, whereas $d(y_{j'},v) < d(y_{j'},u)$.
It implies that the ladder induced by $\{x_0,x_1,\ldots,x_\ell\} \cup \{y_0,y_1,\ldots,y_\ell\}$ is an isometric subgraph of $G$.\qed
\end{proof}

The following observation can be used for proving similar coarse characterizations of the $\alpha_i$-metric property in some classes of $1$-hyperbolic graphs.

\begin{lemma}\label{lem:1hyp}
    Let $u,v,x,y$ be vertices in a graph $G$ such that $u \sim v$, $u \in I(x,v)$, $v \in I(y,u)$, and $d(x,y) \le d(x,u)+d(v,y)-3$.
    If $G$ is $1$-hyperbolic then $S_1(u,x) \subseteq S_1(u,y)$.
\end{lemma}
\begin{proof}
    Let $u' \in S_1(u,x)$ be arbitrary.
    To prove that $u' \in S_1(u,y)$, it suffices to prove that $d(y,u') = d(y,v)$.
    For that, we first note that $u',v$ are nonadjacent, which is because $d(x,u) < d(x,v)$. Therefore, $d(u',v) = 2$.
    Then, let us consider the three distance sums: $(A)$ $d(u',v)+d(x,y) = 2+d(x,y)$, $(B)$ $d(u',x)+d(v,y) = d(x,u)+d(v,y)-1$, and $(C)$ $d(u',y)+d(v,x)$.
    In particular, $(A) \le d(x,u)+d(v,y)-1 = (B)$, and $(C) = d(u',y)+d(x,u)+1 \ge d(y,v)+d(u,x)+1 = (B)+2$.
    Therefore, $(C) > (B) \ge (A)$.
    Since $G$ is $1$-hyperbolic, $(C)-(B) \le 2$.
    As a result, $(C) = (B)+2$, which implies $d(y,u') = d(y,v)$. \qed
\end{proof}

A graph satisfies the \emph{quadrangle condition} $(QC)$ if for every vertices $u,v$ such that $d(u,v) \ge 2$, every two vertices $x,y \in S_1(u,v)$, $x\nsim y$, have a common neighbour in $S_2(u,v)$.
The quadrangle condition is satisfied by all weakly modular graphs, which include the most studied classes in Metric Graph Theory, such as: median graphs (i.e., 1-skeletons of CAT(0) cube complexes), Helly graphs, and bridged graphs~\cite{WeaklyModular}.

\begin{theorem}\label{thm:1hyp-qc}
    If $G=(V,E)$ is $1$-hyperbolic and satisfies the quadrangle condition then, for every integer $i \ge 3$, either $G$ is $\alpha_i$-metric or it contains an isometric ladder of height $\left\lfloor \frac{i-1} 2 \right\rfloor$.   
\end{theorem}
\begin{proof}
    The proof if similar to that of Theorem~\ref{thm:1-slim}.
    Assuming $G$ is not $\alpha_i$-metric, there exist vertices $u,v,x,y$ that satisfy the following conditions: $u \sim v$, $u \in I(x,v)$, $v \in I(y,u)$, and $d(x,y) \le d(x,u)+d(v,y)-i$.
    Set $\ell = \left\lfloor \frac{i-1} 2 \right\rfloor$.
    We fix some shortest $(x,u)$-path, denoting by $x_\ell,x_{\ell-1},\ldots,x_1,x_0=u$ its last $\ell+1$ vertices.
    Then, using $(QC)$, we prove, in what follows, the existence of vertices $y_\ell,y_{\ell-1},\ldots,y_1,y_0$ such that $y_0 = v$, and that the following properties hold for every $1 \le j \le \ell$:
    \begin{itemize}
        \item $x_j \sim y_j$;
        \item and $y_j \in I(y,y_{j-1})$.
    \end{itemize}
    For that, we proceed by finite induction on $j$.
    Let $j$ be such that $1 \le j \le \ell$, and there exists a vertex $y_{j-1}$ such that: $x_{j-1} \sim y_{j-1}$, $y_{j-1} \in I(y,v)$, and $d(y_{j-1},v) = j-1$.
    (Such a vertex $y_{j-1}$ exists if $j=1$, because $y_0 = v$, and otherwise its existence follows from the induction hypothesis).
    Since $d(x,y) \le d(x,u)+d(v,y) - i$, $d(x,y) \le d(x,x_{j-1})+d(y_{j-1},y) + 2(j-1)-i \le d(x,x_{j-1})+d(y_{j-1},y) + 2(\ell-1)-i \le d(x,x_{j-1})+d(y_{j-1},y) -3$.
    Furthermore, $x_{j-1} \in I(x,y_{j-1})$, and $y_{j-1} \in I(y,x_{j-1})$.
    By Lemma~\ref{lem:1hyp}, $x_j \in S_1(x_{j-1},y)$, or equivalently $d(y,x_j) = d(y,y_{j-1})$.
    Therefore, by applying $(QC)$ to $x_j,y_{j-1} \in S_1(x_{j-1},y)$, we get the existence of some common neighbour $y_j \in S_2(x_{j-1},y)$ of $x_j,y_{j-1}$.
    In particular, $x_j \sim y_j$, and $y_j \in I(y,y_{j-1})$.

    By using the same arguments as for the proof of Theorem~\ref{thm:1-slim}, we conclude that $\{x_\ell,x_{\ell-1},\ldots,x_1,x_0\} \cup \{y_\ell,y_{\ell-1},\ldots,y_1,y_0\}$ induces an isometric ladder of $G$ of height $\ell$.
\end{proof}

\section{Graphs with $\alpha_1$-metric}\label{sec:i-1}

By Theorem~\ref{thm:main}, every $\alpha_0$-metric graph must be $1$-hyperbolic, and similarly every $\alpha_1$-metric graph must be $2$-hyperbolic.
In fact, by~\cite[Corollary 30]{WuZh01}, the hyperbolicity of an $\alpha_0$-metric graph is at most $\frac 1 2$.
In this section, we improve the bound on the hyperbolicity of $\alpha_1$-metric graphs from $2$ to $1$ (Theorem~\ref{thm:alpha-1}).
The bound is sharp even for the subclass of chordal graphs~\cite{BKM01}.

    \subsection{Intermediate Results}\label{sec:prelim-i-1}

    First, we recall an elegant characterization of $\alpha_1$-metric graphs.

\begin{theorem} [see Corollary 1 in~\cite{YuCh1991}]\label{th:charact}
$G$ is an  $\alpha_1$-metric graph if and only if all disks $D(v,k)$ $(v\in V$, $k\geq 1)$ of $G$ are convex and $G$ does not contain the graph \WW\ from Fig. \ref{fig:forbid} as an isometric subgraph.
\end{theorem}


  \begin{figure}[htb]
      \centering
      \begin{tikzpicture}
          \node[circle,fill=black,inner sep=0pt,minimum size=5pt] at (-3.5,0) {};
          \node[circle,fill=black,inner sep=0pt,minimum size=5pt] at (-3,.5) {};
          \node[circle,fill=black,inner sep=0pt,minimum size=5pt] at (-2.5,0) {};
          \node[circle,fill=black,inner sep=0pt,minimum size=5pt] at (-2.5,1) {};
          \node[circle,fill=black,inner sep=0pt,minimum size=5pt] at (-2,.5) {};
          \node[circle,fill=black,inner sep=0pt,minimum size=5pt] at (-1.5,0) {};
          \node[circle,fill=black,inner sep=0pt,minimum size=5pt] at (-1.5,1) {};
          \node[circle,fill=black,inner sep=0pt,minimum size=5pt] at (-1,.5) {};
          \node[circle,fill=black,inner sep=0pt,minimum size=5pt] at (-.5,0) {};
          \draw (-3.5,0) -- (-2.5,1) -- (-1.5,1) -- (-.5,0) -- (-3.5,0);
          \draw (-2.5,1) -- (-2,.5) -- (-3,.5) -- (-2.5,0) -- (-2,.5) -- (-1.5,0) -- (-1,.5) -- (-2,.5) -- (-1.5,1);
      \end{tikzpicture}
      \caption{Forbidden isometric subgraph \WW.}
      \label{fig:forbid}
      \vspace*{-5mm}
  \end{figure}

\begin{lemma}[see Theorem 2 in~\cite{SoCh1983}]\label{lem:convex}
All disks $D(v,k)$ $(v\in V$, $k\geq 1)$ of a graph $G$ are convex if and only if $G$ does not contain isometric cycles of length $\ell > 5$, and for any two vertices $x,y$ of $G$ the slice $S_1(x,y)$ is a clique.  
\end{lemma}

The next lemma characterizes the situation in which, in an $\alpha_1$-metric graph, the union of two shortest paths with one common terminal edge is not a shortest path.  

\begin{lemma} [see Lemma 5 in~\cite{BaCh2003}]  \label{lm:1-hyp}
Let $G$ be an $\alpha_1$-metric graph. Let $x,y,v,u$ be vertices of $G$ such that
$v \in I(x,y)$, $x\in I(v,u)$, and $x$ and $v$ are adjacent. Then $d(u,y)=d(u,x) + d(v,y)$ holds if and only if there exist a neighbour $x'$ of $x$ in $I(x,u)$ and a neighbour $v'$ of $v$ in  $I(v,y)$ with $d(x',v')=2$; in particular, $x'$ and $v'$ lie on a common shortest path of $G$ between $u$ and $y$.
\end{lemma}

We will also need in our proofs the following properties of metric triangles in $\alpha_1$-metric graphs.

\begin{lemma}[see Lemma 4 in~\cite{BaCh2003}]\label{lem:triangle}
In an $\alpha_1$-metric graph $G$, every metric triangle is of type $(1,1,1)$, $(1,2,2)$, $(2,1,2)$, $(2,2,1)$ or $(2,2,2)$.
\end{lemma}

Finally, we conclude this subsection with two more technical results. We note that they could be also deduced from~\cite{chalopin2022graphs}.

\begin{lemma}[see Corollary 6 in~\cite{DrDu23}]\label{lem:C3orC5}
Let $G$ be an $\alpha_1$-metric graph. Then, for every edge $xy\in E$ and a vertex $u\in V$ with $d(u,x)=d(u,y)=k$,  either there is a common neighbour $u'$ of $x$ and $y$ at distance $k-1$ from $u$ or there exists a vertex $u'$ at distance 2 from $x$ and $y$ and at distance $k-2$ from $u$ such that, for every $z\in N(x)\cap N(u')$ and $w\in N(y)\cap N(u')$, the sequence $(x,z,u',w,y)$ forms an induced $C_5$ in $G$.
\end{lemma}

\begin{corollary}\label{cor:c3}
Let $x,y \in S_k(u,v)$ be adjacent in an $\alpha_1$-metric graph $G$.
Then, $x,y$ have common neighbours in both $S_{k-1}(u,v)$ and $S_{k+1}(u,v)$.
\end{corollary}
\begin{proof}
In what follows, we only prove that $x,y$ have a common neighbour in $S_{k-1}(u,v)$ (the other part of the proof is symmetric).
Suppose, for the sake of contradiction, that it is not the case.
By Lemma~\ref{lem:C3orC5}, there exists a vertex $w$ such that $d(x,w) = d(y,w) = 2$ and $d(u,w) = k-2$.
Note that in this situation, $w \in S_{k-2}(u,v)$ and $x,y \in S_2(w,v)$.
Let $x' \in N(x) \cap N(w)$ and let $y' \in N(y) \cap N(w)$.
Again, by Lemma~\ref{lem:C3orC5}, the sequence $(x,x',w,y',y)$ induces a $C_5$.
It implies that $x',y' \in S_1(w,v)$ are nonadjacent, thus implying that not all the disks of $G$ are convex by Lemma~\ref{lem:convex}.
More specifically, $D(v,d(v,w)-1)$ is not convex, which is because $x',y' \in D(v,d(v,w)-1)$ but $w \in I(x',y') \setminus D(v,d(v,w)-1)$.
However, the latter contradicts Theorem~\ref{th:charact}. \qed
\end{proof}

    \subsection{Hyperbolicity}\label{sec:hyp-i-1}

Our strategy in order to prove Theorem~\ref{thm:alpha-1} is essentially the same as what we did for Theorem~\ref{thm:main}: we upper bound the interval thinness of the injective hull. For that, we heavily use structural and metric properties of $\alpha_1$-metric graphs. 

We first improve Lemma~\ref{lem:three-balls}, but only for the $\alpha_1$-metric graphs.

\begin{lemma}\label{lem:alpha-1-thinness}
Let $D(u,r_u),D(v,r_v),D(w,r_w)$ be pairwise intersecting disks of $G$ such that $d(u,v) = r_u + r_v$.
If $G$ is an $\alpha_1$-metric graph, then {\em every} vertex $x \in S_{r_u}(u,v) (= S_{r_v}(v,u))$ satisfies $d(w,x) \leq r_w + 2$.
\end{lemma}

\begin{proof}
Suppose, by way of contradiction, that $d(x,w) > r_w + 2$ for some $x \in S_{r_u}(u,v)$.

Let $u'v'w'$ be a quasi median for the triple $u,v,w$.
If $d(u,u') \geq r_u$, then we consider a vertex $x_u \in S_{r_u}(u,u')$.
Note that $x_u \in I(u,v) \cap I(u,w)$.
It implies $d(x_u,v) = r_v$ and $d(x_u,w) \leq r_w$.
By Lemma~\ref{lem:thinness} and as $x,x_u \in S_{r_u}(u,v)$, $d(x,x_u) \leq 2$, and so $d(x,w) \leq r_w + 2$, leading to a contradiction.
In the same way, if $d(u,v') \leq r_u$, then (up to reverting the respective roles of $u$ and $v$) we can prove similarly as above that $d(x,w) \leq r_w + 2$.
Thus, from now on, we assume that $d(u,u') < r_u < d(u,v')$.
By Lemma~\ref{lem:triangle}, it implies $d(u',v') = 2$, and so $d(u,u') = r_u - 1, \ d(v,v') = r_v - 1$.
Let $y \in N(u') \cap N(v')$ be arbitrary. 
Since we have $u',v' \in D(w,r_w+1)$, $y \in I(u',v')$, and as by Theorem~\ref{th:charact} the disk $D(w,r_w+1)$ is convex, $d(y,w) \leq r_w + 1$.
If $d(y,w) \leq r_w$, then again, by Lemma~\ref{lem:thinness}, $d(x,w) \leq d(x,y) + d(y,w) \leq r_w + 2$, giving a contradiction.
Therefore, from now on, $d(y,w) = r_w + 1$.

We construct an isometric copy of \WW as follows (see also Fig.~\ref{fig:forbidden-construct-1}).

\begin{figure}[!ht]
    \centering
    \begin{tikzpicture}

        \draw[dashed] (-5,1) -- (-1,2); \draw[dashed] (-5,1) -- (-1,1); \draw[dashed] (-5,1) -- (-1,0);
        \draw[dashed] (5,1) -- (1,2); \draw[dashed] (5,1) -- (1,1); \draw[dashed] (5,1) -- (2,0);
        \draw[dashed] (0,5) -- (0,3.5);
        \draw (-1,2) -- (1,2) -- (0,3.5) -- (-1,2); \draw (0,2) -- (0,0);
        \draw (-1,2) -- (-1,0) -- (1,0) -- (1,2);
        \draw (1,0) -- (2,0) -- (1,1);
        \draw (-1,1) -- (0,2) -- (1,1);
        \draw (-1,1) -- (1,1);
        \draw (-1,0) -- (0,1) -- (1,0);
        
        \draw[thick,red] (-1,2) -- (-1,0) -- (2,0) -- (0,2) -- (-1,2);
        \draw[thick, red] (1,0) -- (1,1); \draw[thick,red] (0,2) -- (-1,1);
        \draw[thick,red] (-1,1) -- (1,1); \draw[thick,red] (0,0) -- (0,2);
        \draw[thick,red] (-1,0) -- (0,1) -- (1,0);

        \node[circle,fill=black,inner sep=0pt,minimum size=5pt, label=below:{$x$}] at (0,0) {};
        \node[circle,fill=black,inner sep=0pt,minimum size=5pt, label=above:{$y$}] at (0,2) {};
        \node[circle,fill=black,inner sep=0pt,minimum size=5pt, label=above:{$u'$}] at (-1,2) {};
        \node[circle,fill=black,inner sep=0pt,minimum size=5pt, label=above:{$v'$}] at (1,2) {};
        \node[circle,fill=black,inner sep=0pt,minimum size=5pt, label=right:{$w'$}] at (0,3.5) {};
        \node[circle,fill=black,inner sep=0pt,minimum size=5pt, label=above:{}] at (-.5,2.75) {};
        \node[circle,fill=black,inner sep=0pt,minimum size=5pt, label=above:{}] at (.5,2.75) {};
        \node[circle,fill=black,inner sep=0pt,minimum size=5pt, label=above:{$w$}] at (0,5) {};
        \node[circle,fill=black,inner sep=0pt,minimum size=5pt, rotate=-45, label=left:{$z$}] at (0,1) {};
        \node[circle,fill=black,inner sep=0pt,minimum size=5pt, label=below:{$a_u$}] at (-1,0) {};
        \node[circle,fill=black,inner sep=0pt,minimum size=5pt, rotate=-45, label=left:{$b_u$}] at (-1,1) {};
        \node[circle,fill=black,inner sep=0pt,minimum size=5pt, label=below:{$a_v$}] at (1,0) {};
        \node[circle,fill=black,inner sep=0pt,minimum size=5pt, rotate=45, label=right:{$b_v$}] at (1,1) {};
        \node[circle,fill=black,inner sep=0pt,minimum size=5pt, label=below:{$c_v$}] at (2,0) {};
        \node[circle,fill=black,inner sep=0pt,minimum size=5pt, label=below:{$u$}] at (-5,1) {};
        \node[circle,fill=black,inner sep=0pt,minimum size=5pt, label=below:{$v$}] at (5,1) {};

    \end{tikzpicture}
    \caption{To the proof of Lemma~\ref{lem:alpha-1-thinness}: construction of a forbidden \WW.}
    \label{fig:forbidden-construct-1}
\end{figure}

\begin{itemize}
    \item By Lemma~\ref{lem:thinness}, $d(x,w) \leq 2 + d(y,w) = r_w + 3$.
          Therefore, $d(x,y) = 2$ and $d(x,w) = r_w + 3$.
          Let $z \in N(x) \cap N(y)$.
          As $x,y \in S_{r_u}(u,v)$, and as by Theorem~\ref{th:charact} the disks $D(u,r_u),D(v,r_v)$ are convex, $z \in S_{r_u}(u,v)$.
    \item Let $a_u \in N(x) \cap S_{r_u-1}(u,v)$.
          By Lemma~\ref{lem:thinness}, $d(a_u,u') \leq 2$.
          We prove, as an intermediate claim, that $d(x,u') = 3$ (and therefore, $d(a_u,u') = 2$).
          Indeed, if it was not the case, then we would obtain $d(x,u') = 2$ and $d(u',w) = r_w + 1$. In particular, $u' \in I(x,w)$.
          We fix a shortest $(u',x)$-path, and we note that it contains an edge $st$ such that $d(u',s) < d(u',t), \ d(v,s) = r_v+1$ and $d(v,t) = r_v$ (such an edge exists because $d(v,u') = r_v+1$, $d(v,x) = r_v$, and $d(u',x) = 2$).
          By the $\alpha_1$-metric property applied to $w,s,t,v$, $d(v,w) \geq d(v,t) + d(s,w) \geq d(v,t) + d(u',w) = r_v + r_w + 1$, giving a contradiction.
          As a result, $d(x,u') = 3$ and $d(a_u,u') = 2$.
          Let $b_u \in N(u') \cap N(a_u)$.
          As $u',a_u \in S_{r_u-1}(u,v)$, and as by Theorem~\ref{th:charact} the disks $D(u,r_u-1),D(v,r_v+1)$ must be convex, $b_u \in S_{r_u-1}(u,v)$.
    \item Let $a_v \in N(x) \cap S_{r_u+1}(u,v)$.
          We prove as before that $d(v',x) = 3$ and $d(v',a_v) = 2$.
          Let $b_v \in N(v') \cap N(a_v)$.
          As $v',a_v \in S_{r_u+1}(u,v)$, and as by Theorem~\ref{th:charact} the disks $D(u,r_u+1),D(v,r_v-1)$ are convex, $b_v \in S_{r_u+1}(u,v)$.
    \item Last, let $c_v \in N(a_v) \cap N(b_v) \cap S_{r_u+2}(u,v)$, that exists by Corollary~\ref{cor:c3}.
\end{itemize}

\begin{claimf}
    $X = \{u',a_u,b_u,x,y,z,a_v,b_v,c_v\}$ induces a \WW.
\end{claimf} 
Since $b_u,y \in S_1(u',x)$, by Lemma~\ref{lem:convex}, we have $b_uy \in E$.
In the same way, $b_vy \in E$.
Hence, there is a cycle with vertex-set $\{x,a_v,b_v,y,b_u,a_u\}$.
Furthermore $d(x,b_v) = d(x,y) = d(x,b_u) = 2$, and also there is no edge between $\{a_u,b_u\}$ and $\{a_v,b_v\}$. 
Suppose, by way of contradiction, that $ya_u \in E$. Then, $a_u \in I(x,y) \setminus D(v,r_v)$, which implies $D(v,r_v)$ is not convex, thus contradicting Theorem~\ref{th:charact}.
Therefore, $ya_u \notin E$, and we can prove similarly $ya_v \notin E$.
Doing so, we obtain that the $C_6$ with vertex-set $\{x,a_v,b_v,y,b_u,a_u\}$ is induced.
Next, we prove that $z$ is adjacent to every vertex of $\{x,a_v,b_v,y,b_u,a_u\}$.
For $x,y$ that follows from the definition of $z$ and for $a_u,a_v$ that follows from Lemma~\ref{lem:convex} because $z,a_u \in S_1(x,u')$ and $z,a_v \in S_1(x,v')$.
Now, there is a $C_4$ with vertex-set $\{z,a_u,b_u,y\}$. Since we already proved that $ya_u \notin E$, by Lemma~\ref{lem:convex}, we obtain $zb_u \in E$.
In the same way, by considering a $C_4$ with vertex-set $\{z,a_v,b_v,y\}$, we obtain that $zb_v \in E$.
Doing so, we get that  $\{x,a_v,b_v,y,b_u,a_u,z\}$ induces a copy of the wheel $W_6$.
By the choice of $c_v$, its only neighbours in $X$ are $a_v,b_v$.
Last, since $u' \in S_{r_u-1}(u,v)$ and $d(u',x) = 3$, we obtain that the only neighbours of $u'$ in $X$ are $b_u,y$.
So, we proved as claimed that $X$ induces a \WW. {\hfill $\diamond$}

\smallskip
It remains to prove that $G[X]$ is an isometric subgraph, thus directly contradicting Theorem~\ref{th:charact}.
The only pairs of vertices that are at distance $3$ in $G[X]$ are: $(u',x)$, $(u',a_v)$, $(u',c_v)$, $(a_u,c_v)$, $(b_u,c_v)$.
Note that we already proved that $d(u',x) = 3$. 
Furthermore, $d(u',a_v) > 2$ because otherwise $a_v \notin D(u,r_u)$ would be on a shortest $(u',x)$-path, thus contradicting Theorem~\ref{th:charact} because $D(u,r_u)$ would not be convex.
Finally, $\min\{d(u',c_v),d(a_u,c_v),d(b_u,c_v)\} \geq 3$ because $u',a_u,b_u \in S_{r_u-1}(u,v)$ and $c_v \in S_{r_u+2}(u,v)$. \qed
\end{proof}

We need one more technical lemma.

\begin{lemma}\label{lem:close-balls}
Let $u,v,a,b$ be vertices of a graph $G=(V,E)$ such that:
\begin{itemize}
    \item $d(u,v) = r_u + r_v$;
    \item $d(u,a) \leq r_u + r_a, \ d(v,a) \leq r_v + r_a$;
    \item $d(u,b) \leq r_u + r_b, \ d(v,b) \leq r_v + r_b$.
\end{itemize}
If $G$ is an $\alpha_1$-metric graph, then $d(a,b) \leq r_a + r_b + 2$.
\end{lemma}

\begin{proof}
Suppose, by way of contradiction, that $d(a,b) > r_a + r_b + 2$.
The reader can follow our construction using Fig.~\ref{fig:forbidden-construct-2}.
Our proof, in what follows, is divided into several intermediate claims.

\begin{claimf}\label{claim:1}
    $d(a,b) = r_a+r_b+3$. 
    Furthermore, there exist $x,y \in I(a,b) \cap S_{r_u}(u,v)$ such that $x \sim y$, $d(a,x) = r_a+1$, and $d(b,y) = r_b+1$.
\end{claimf}
By Lemma~\ref{lem:three-balls}, there exists a vertex $x \in S_{r_u}(u,v)$ such that $d(a,x) \leq r_a + 1$.
By Lemma~\ref{lem:alpha-1-thinness}, $d(x,b) \leq r_b + 2$.
As a result, $d(a,b) = r_a + r_b + 3$.
Furthermore, $x \in I(a,b)$.
Let $y \in S_1(x,b)$ be arbitrary.
If $d(u,y) > r_u$ then, by the $\alpha_1$-metric property applied to $u,x,y,b$, we obtain $d(u,b) \geq d(u,x) + d(y,b) = r_u + r_b +1$, and a contradiction with the assumption of the lemma arises.
In the same way, if $d(v,y) > r_v$ then, by the $\alpha_1$-metric property applied to $v,x,y,b$, we obtain $d(v,b) \geq r_v + r_b + 1$, giving a contradiction again.
As a result, $y \in S_{r_u}(u,v)$.{\hfill $\diamond$}

\begin{claimf}\label{claim:2}
    The slices $S_{r_u}(u,v)$, $S_1(x,a)$, and $S_1(y,b)$ are pairwise disjoint.        
\end{claimf}
As by Claim~\ref{claim:1}, $xy$ is an edge of some shortest $(a,b)$-path, $S_1(x,a)$, and $S_1(y,b)$ are disjoint.
We only prove, in what follows, that $S_{r_u}(u,v)$ and $S_1(y,b)$ are also disjoint (the proof that $S_{r_u}(u,v)$ and $S_1(x,a)$ are also disjoint is obtained with symmetric reasoning).
For that, let $z \in S_1(y,b)$ be arbitrary.
Suppose, for the sake of contradiction, $z \in S_{r_u}(u,v)$.
By Lemma~\ref{lem:alpha-1-thinness}, $d(z,a) \leq r_a + 2$, and so $d(a,b) \leq d(a,z) + d(z,b) \leq r_a + r_b + 2$, a contradiction.{\hfill $\diamond$}

\begin{claimf}\label{claim:3}
    For any vertex $p \in \{u,v\}$ such that $S_1(y,b) \not\subseteq D(p,r_p)$, $d(p,b) = r_p+r_b$.
    Furthermore, there exist $w_p,w_{pb} \in N(y)$ such that: $w_p \sim w_{pb}$, $d(p,w_{pb}) = d(p,w_p)+1 = r_p$, and $d(w_{pb},b) = r_b$.
\end{claimf}
Let $z \in S_1(y,b)$ be satisfying $d(z,p) > r_p$. As $y \in S_{r_u}(u,v)$, $d(z,p) = r_p+1$.
By the $\alpha_1$-metric property applied to $p,y,z,b$, $d(p,b) \geq d(p,y) + d(z,b) = r_p + r_b$. 
This together with the assumption $d(p,b) \leq r_p + r_b$ implies $d(p,b) = d(p,y) + d(z,b) = r_p + r_b$. 
By Lemma~\ref{lm:1-hyp}, there exist $w_p \in S_1(y,p)$ and $w_b \in S_1(z,b)$ satisfying $d(w_p,w_b) = 2$.
By Corollary~\ref{cor:c3}, there also exists some vertex $c \in S_1(x,p) \cap S_1(y,p)$.
Among all the possibilities for the triple $w_p,w_b,c$, we choose one such that $d(w_p,c)$ is minimized.
Doing so, if $d(c,w_b) = 2$, then as $c \in S_1(y,p)$, we have $c = w_p$.

Let $w_{pb} \in N(w_p) \cap N(w_b)$.
To prove the claim, it suffices to prove that $w_{pb} \sim y$.
For that, we start proving that $w_{pb} \sim z$. 
This is done by proving that either $w_{pb},z \in S_1(w_b,a)$ or $w_{pb},z \in S_1(w_b,c)$, and then applying Lemma~\ref{lem:convex}. 
Since $c,w_p \in S_1(y,p)$, by Lemma~\ref{lem:convex}, either $c = w_p$ or $cw_p \in E$.
We consider both cases separately.
\begin{figure}[!ht]
    \centering
    \begin{tikzpicture}

            \node[circle,fill=black,inner sep=0pt,minimum size=5pt, label=above:{$u$}] at (0,5) {};
            \node[circle,fill=black,inner sep=0pt,minimum size=5pt, label=below:{$v$}] at (0,-5) {};
            \node[circle,fill=black,inner sep=0pt,minimum size=5pt, label=left:{$a$}] at (-8,0) {};
            \node[circle,fill=black,inner sep=0pt,minimum size=5pt, label=right:{$b$}] at (8,0) {};

            \node[circle,fill=black,inner sep=0pt,minimum size=5pt, rotate=45,label=below:{$x$}] at (-1,0) {};
            \node[circle,fill=black,inner sep=0pt,minimum size=5pt, rotate=-45, label=below:{$y$}] at (1,0) {};
            \node[circle,fill=black,inner sep=0pt,minimum size=5pt, label=above:{$z$}] at (2,-.2) {};

            \node[circle,fill=black,inner sep=0pt,minimum size=5pt, rotate=-45,label=right:{$w_u$}] at (1,1) {};
            \node[circle,fill=black,inner sep=0pt,minimum size=5pt, rotate=45, label=below:{$w_b$}] at (3.5,0) {};
            \node[circle,fill=black,inner sep=0pt,minimum size=5pt, label=above:{$w_{ub}$}] at (3,1) {};

            \node[circle,fill=black,inner sep=0pt,minimum size=5pt, label=below:{$c$}] at (0,1) {};

            \node[circle,fill=black,inner sep=0pt,minimum size=5pt, rotate=-45,label=right:{$w_v$}] at (1,-1) {};
            \node[circle,fill=black,inner sep=0pt,minimum size=5pt, rotate=-45,label=right:{$w_{vb}$}] at (3,-1) {};

            \node[circle,fill=black,inner sep=0pt,minimum size=5pt, rotate=45,label=right:{$w_u'$}] at (-1,1) {};
            \node[circle,fill=black,inner sep=0pt,minimum size=5pt, label=above:{$w_{ua}'$}] at (-3,1) {};
            \node[circle,fill=black,inner sep=0pt,minimum size=5pt, rotate=-45,label=right:{$w_v'$}] at (-1,-1) {};
            \node[circle,fill=black,inner sep=0pt,minimum size=5pt, rotate=-45,label=right:{$w_{va}'$}] at (-3,-1) {};

            \node at (-4,.2) {$r_a+1$};
            \node at (-6,.6) {$r_a$};
            \node at (-6,-.6) {$r_a$};
            \node at (5,.2) {$r_b-1$};
            \node at (6,.6) {$r_b$};
            \node at (6,-.6) {$r_b$};
            \node at (0,2) {$r_u-1$};
            \node at (-1.1,3) {$r_u-1$};
            \node at (1.1,3) {$r_u-1$};
            \node at (-1.1,-3) {$r_v-1$};
            \node at (1.1,-3) {$r_v-1$};

            \draw[dashed] (-8,0) -- (-1,0);
            \draw (-1,0) -- (1,0);
            \draw (1,0) -- (2,-.2);
            \draw[dashed] (0,5) -- (1,1); \draw[dashed] (3.5,0) -- (8,0);
            \draw (1,1) -- (3,1) -- (3.5,0);
            \draw (1,1) -- (1,0); \draw (2,-.2) -- (3.5,0);
            \draw (-1,0) -- (0,1) -- (1,0);
            \draw[dashed] (0,5) -- (0,1);
            \draw (0,1) -- (1,1);
            \draw (1,0) -- (3,1);
            \draw[dashed] (0,-5) -- (1,-1);
            \draw (1,-1) -- (1,0);
            \draw (1,-1) -- (3,-1);
            \draw (1,0) -- (3,-1);
            \draw[dashed] (0,5) -- (-1,1); \draw[dashed] (0,-5) -- (-1,-1);
            \draw (-3,1) -- (-1,1) -- (-1,-1) -- (-3,-1);
            \draw (-3,1) -- (-1,0) -- (-3,-1);
            \draw (0,1) -- (-1,1);

            \draw[dashed] (-8,0) -- (-3,-1); \draw[dashed] (-8,0) -- (-3,1);
            \draw[dashed] (8,0) -- (3,1); \draw[dashed] (8,0) -- (3,-1);

            \draw (-3,-1) -- (-3,1);
            \draw (3,-1) -- (3,1);

    \end{tikzpicture}
    \caption{To the proof of Lemma~\ref{lem:close-balls}.}
    \label{fig:forbidden-construct-2}
\end{figure}
\begin{itemize}
    \item {\it Case $c = w_p$}. Recall that $d(b,w_p) = r_b + 1$. Furthermore, $d(a,w_p) \leq 1 + d(a,x) = r_a + 2$. Therefore, $w_p,w_{pb} \in I(w_b,a) \subseteq I(b,a)$. It implies $w_{pb},z \in S_1(w_b,a)$. By Lemma~\ref{lem:convex}, $w_{pb}z \in E$. 
    \item {\it Case $cw_p \in E$}.  Note that $d(c,w_b) > 1$ (else, $d(p,b) \leq d(p,c) + 1 + d(w_b,b) = r_p - 1 + 1 + r_b - 1 < r_p + r_b$). In particular, due to our choice for the triple $w_p,w_b,c$, $d(w_b,c) > 2$. But then $d(c,w_b) = 3$ and $z,w_{pb} \in S_1(w_b,c)$. By Lemma~\ref{lem:convex}, we obtain $w_{pb}z \in E$.  
\end{itemize}
Summarizing, in both cases we get $w_{pb}z \in E$. Hence, we obtain that there is a $C_4$ with vertex-set $\{w_{pb},z,y,w_p\}$. Since $w_pz \notin E$ (else, $d(p,z) \leq r_p$), by Lemma~\ref{lem:convex}, we have $w_{pb} \sim y$. 
{\hfill $\diamond$}

\begin{claimf}\label{claim:4}
    For any vertex $q \in \{u,v\}$ such that $S_1(x,a) \not\subseteq D(q,r_q)$, $d(q,a) = r_q+r_a$.
    Furthermore, there exist $w_q',w_{qa}' \in N(x)$ such that: $w_q' \sim w_{qa}'$, $d(q,w_{qa}') = d(q,w_q')+1 = r_q$, and $d(w_{qa},a) = r_a$.
\end{claimf}
To prove the claim, it suffices to reproduce the same reasoning as in the proof of Claim~\ref{claim:3}, where we replace $p,a,b,x,y$ with $q,b,a,y,x$.{\hfill $\diamond$}

\medskip
Let us put Claims~\ref{claim:2},~\ref{claim:3},~\ref{claim:4} altogether.
By Claim~\ref{claim:2}, either $S_1(y,b) \not\subseteq D(u,r_u)$ or $S_1(y,b) \not\subseteq D(v,r_v)$.
Without loss of generality, let us assume that $S_1(y,b) \not\subseteq D(u,r_u)$.
By Claim~\ref{claim:3}, applied for $p=u$, the following holds:
\begin{itemize}
    \item $d(u,b) = r_u+r_b$;
    \item there exist adjacent vertices $w_u,w_{ub} \in N(y)$ such that $d(u,w_{ub}) = d(u,w_u)+1=r_u$ and $d(w_{ub},b)=r_b$.
\end{itemize}
As $d(y,b)=r_b+1$ and $y \sim w_{ub}$, we obtain $w_{ub} \in S_1(y,b)$.
By Claim~\ref{claim:2},  $w_{ub} \notin S_{r_u}(u,v)$.
Therefore, as $d(u,w_{ub}) = r_u$, $d(v,w_{ub}) > r_v$.
By Claim~\ref{claim:3}, applied for $p=v$, we obtain the following information: 
\begin{itemize}
    \item $d(v,b) = r_v + r_b$;
    \item there exist adjacent vertices $w_v,w_{vb} \in N(y)$ such that $d(v,w_{vb}) = d(v,w_v) + 1 = r_v$ and $d(w_{vb},b) = r_b$. 
\end{itemize}

We reproduce the same reasoning as above, where we replace $b$ with $a$ and Claim~\ref{claim:3} with Claim~\ref{claim:4}. Doing so, we obtain the following additional information:
\begin{itemize}
    \item $d(u,a) = r_u + r_a$; $d(v,a) = r_v + r_a$;
    \item there exist adjacent vertices $w_u',w_{ua}' \in N(x)$  such that  $d(u,w_{ua}') = d(u,w_u') + 1 = r_u$ and $d(w_{ua}',a) = r_a$;
    \item there exist adjacent vertices $w_v',w_{va}' \in N(x)$ such that $d(v,w_{va}') = d(v,w_v') + 1 = r_v$ and $d(w_{va}',a) = r_a$.
\end{itemize}
By Lemma~\ref{lem:convex}, we have that $w_{ub},w_{vb} \in S_1(y,b)$ are adjacent, and similarly $w_{ua}',w_{va}' \in S_1(x,a)$ are adjacent.
Furthermore, $w_{u} \neq w_{u}'$ (otherwise, $d(a,b) \leq d(a,w_{u}') + d(w_{u},b) = r_a + r_b + 2$) and, similarly, $w_v \neq w_v'$.
In what follows, let us assume $d(w_u,w_u') + d(w_v,w_v')$ to be minimized.
By Lemma~\ref{lem:thinness}, $1 \leq d(w_u,w_u'), d(w_v,w_v') \leq 2$.
We need to distinguish between several cases.

\bigskip
\noindent
\textit{\textbullet~Case $d(w_u,w_u') = 2$ or $d(w_v,w_v') = 2$.}
See Fig.~\ref{fig:forbidden-construct-2-case-1}.
\medskip

\begin{figure}[!ht]
    \centering
    \begin{tikzpicture}

            \draw (-1,0) -- (1,0);
            \draw (1,1) -- (3,1);
            \draw (1,1) -- (1,0); 
            \draw (1,0) -- (3,1);
            \draw (1,-1) -- (1,0);
            \draw (1,-1) -- (3,-1);
            \draw (1,0) -- (3,-1);
            \draw (-3,1) -- (-1,1) -- (-1,-1) -- (-3,-1);
            \draw (-3,1) -- (-1,0) -- (-3,-1);
            \draw (-3,-1) -- (-3,1);
            \draw (3,-1) -- (3,1);
            \draw (-1,0) -- (0,1) -- (1,0); \draw (-1,1) -- (1,1);
            \draw (-1,1) -- (-.5,2) -- (.5,2) -- (1,1); \draw (-.5,2) -- (0,1) -- (.5,2); 
            \draw (-.5,2) -- (0,3) -- (.5,2);

            \draw[thick,red] (0,3) -- (-1,1) -- (-3,1) -- (-1,0) -- (1,0) -- (1,1) -- (0,3);
            \draw[thick, red] (-1,1) -- (-1,0); \draw[thick,red] (-.5,2) --(.5,2);
            \draw[thick, red] (-1,1) -- (1,1); \draw[thick,red] (-.5,2) -- (0,1) -- (.5,2); \draw[thick,red] (-1,0) -- (0,1) -- (1,0);

            \node[circle,fill=black,inner sep=0pt,minimum size=5pt, rotate=45,label=below:{$x$}] at (-1,0) {};
            \node[circle,fill=black,inner sep=0pt,minimum size=5pt, rotate=-45, label=below:{$y$}] at (1,0) {};
            \node[circle,fill=black,inner sep=0pt,minimum size=5pt, rotate=-45,label=right:{$w_u$}] at (1,1) {};
            \node[circle,fill=black,inner sep=0pt,minimum size=5pt, label=above:{$w_{ub}$}] at (3,1) {};
            \node[circle,fill=black,inner sep=0pt,minimum size=5pt, rotate=-45,label=right:{$w_v$}] at (1,-1) {};
            \node[circle,fill=black,inner sep=0pt,minimum size=5pt, rotate=-45,label=right:{$w_{vb}$}] at (3,-1) {};
            \node[circle,fill=black,inner sep=0pt,minimum size=5pt, rotate=-45,label=below:{$w_u'$}] at (-1,1) {};
            \node[circle,fill=black,inner sep=0pt,minimum size=5pt, label=above:{$w_{ua}'$}] at (-3,1) {};
            \node[circle,fill=black,inner sep=0pt,minimum size=5pt, rotate=-45,label=right:{$w_v'$}] at (-1,-1) {};
            \node[circle,fill=black,inner sep=0pt,minimum size=5pt, rotate=-45,label=right:{$w_{va}'$}] at (-3,-1) {};
            \node[circle,fill=black,inner sep=0pt,minimum size=5pt, label=above:{$c$}] at (0,1) {};
            \node[circle,fill=black,inner sep=0pt,minimum size=5pt, label=left:{$t$}] at (-.5,2) {};
            \node[circle,fill=black,inner sep=0pt,minimum size=5pt, label=right:{$s$}] at (.5,2) {};
            \node[circle,fill=black,inner sep=0pt,minimum size=5pt, label=above:{$w_{s,t}$}] at (0,3) {};

    \end{tikzpicture}
    \caption{Case $d(w_u,w_u') = 2$.}
    \label{fig:forbidden-construct-2-case-1}
\end{figure}

\noindent
By symmetry, let $d(w_u,w_u') = 2$.
By Theorem~\ref{th:charact}, $w_ux,w_u'y \notin E$ (else, $D(u,r_u-1)$ would not be convex).
Let $c \in N(x) \cap N(y) \cap S_{r_u-1}(u,v)$, which exists by Corollary~\ref{cor:c3}.
Since we have $w_u,c \in S_1(y,u)$ and $w_u',c \in S_1(x,u)$, by Lemma~\ref{lem:convex}, $w_uc, w_u'c \in E$.
Let $s \in N(w_u) \cap N(c) \cap S_{r_u-2}(u,v)$ and $t \in N(w_u') \cap N(c) \cap S_{r_u-2}(u,v)$;  they exist by Corollary~\ref{cor:c3}.
If $s=t$ then $s \in I(w_u,w_u') \setminus D(v,r_v+1)$, and as $w_u,w_u' \in D(v,r_v+1)$ the latter would contradict the convexity of $D(v,r_v+1)$ (Theorem~\ref{th:charact}).
Therefore, $s \ne t$.
Since also $c \in N(s) \cap N(t), \ c \notin D(u,r_u-2)$ but $s,t \in D(u,r_u-2)$, by Theorem~\ref{th:charact} and as $D(u,r_u-2)$ is convex, we get $st \in E$.
Furthermore, $sw_u',tw_u \notin E$ (else, $D(v,r_v+1))$ would not be convex).
Then, we obtain that $\{x,y,w_u,s,t,w_u'\}$ induces a $C_6$ and so $\{c,x,y,w_u,s,t,w_u'\}$ induces a wheel $W_6$.

\medskip
Since we have $d(w_{ua}',w_{ub}) = 3$, one of the edges $w_{ua}'c,w_{ub}c$ must be missing.
By symmetry, let us assume $w_{ua}',c$ are non-adjacent.

\begin{claimf}
    $d(w_{ua}',w_u) = 3$.
\end{claimf}
Suppose $d(w_{ua}',w_u) < 3$.
Then, $d(w_{ua}',w_u) = 2$ (else, $d(a,b) \leq d(a,w_{ua}') + 1 + d(w_u,b) = r_a + r_b + 2$).
Let $w' \in N(w_{ua}') \cap N(w_u)$ be arbitrary. See Fig.~\ref{fig:claim-dist_wua-wu}.
Observe that $(w_{ua}',w',w_u,w_{ub})$ and $(w_{ua}',x,y,w_{ub})$ are shortest paths.
Therefore, by Lemma~\ref{lem:convex}, $xw' \in E$.
Note that in this situation, there are non-adjacent vertices $x,w_u \in N(w') \cap N(c)$, and therefore also by Lemma~\ref{lem:convex}, $w'c \in E$.
But then, there are non-adjacent vertices $w_{ua}',c \in N(w') \cap N(w_u')$, and so again, by Lemma~\ref{lem:convex}, $w_u'w' \in E$.
Since $D(u,r_u-1)$ must be convex (Theorem~\ref{th:charact}) and $w' \in N(w_u) \cap N(w_u')$, $d(u,w') = r_u-1$.
We replace $w_u'$ by $w'$, thus contradicting the minimality of $d(w_u,w_u') + d(w_v,w_v')$. {\hfill $\diamond$}
\begin{figure}[!h]
    \centering
    \begin{tikzpicture}

        \draw[thick,red] (-3,1) -- (-.5,1.75) -- (1,1);
        \draw[red] (0,1) -- (-.5,1.75) -- (-1,0); \draw[red] (-.5,1.75) -- (-1,1);

        \node[circle,fill=black,inner sep=0pt,minimum size=5pt, rotate=45,label=below:{$x$}] at (-1,0) {};
        \node[circle,fill=black,inner sep=0pt,minimum size=5pt, rotate=-45, label=below:{$y$}] at (1,0) {};
        \node[circle,fill=black,inner sep=0pt,minimum size=5pt, rotate=-45,label=right:{$w_u$}] at (1,1) {};
        \node[circle,fill=black,inner sep=0pt,minimum size=5pt, label=above:{$w_{ub}$}] at (3,1) {};
        \node[circle,fill=black,inner sep=0pt,minimum size=5pt, rotate=-45,label=below:{$w_u'$}] at (-1,1) {};
        \node[circle,fill=black,inner sep=0pt,minimum size=5pt, label=above:{$w_{ua}'$}] at (-3,1) {};
        \node[circle,fill=black,inner sep=0pt,minimum size=5pt, label=above:{$c$}] at (0,1) {};
        \node[circle,fill=red,inner sep=0pt,minimum size=5pt, label=above:{\textcolor{red}{$w'$}}] at (-.5,1.75) {};

        \draw (-3,1) -- (1,1); \draw[thick] (1,1) -- (3,1);
        \draw[thick] (-3,1) -- (-1,0) -- (1,0) -- (3,1);
        \draw (-1,1) -- (-1,0) -- (0,1) -- (1,0) -- (1,1);

    \end{tikzpicture}
    \caption{To the proof of $d(w_{ua}',w_u) = 3$.}
    \label{fig:claim-dist_wua-wu}
    \vspace*{-5mm}
\end{figure}
\begin{claimf}
    $d(w_{ua}',s) = 3$.
\end{claimf}
Consider the edge $yw_u$.
We have $d(w_{ua}',y) < d(w_{ua}',w_u)$ and $d(s,w_u) < d(s,y)$.
By the $\alpha_1$-metric property applied to $w_{ua}',y,w_u,s$, $d(w_{ua}',s) \geq d(w_{ua}',y) + d(w_u,s) = 2 + 1 = 3$. {\hfill $\diamond$}

\medskip
Let $w_{st} \in N(s) \cap N(t) \cap S_{r_u-3}(u,v)$, that exists by Corollary~\ref{cor:c3}.
It follows from the above that the subgraph induced by $\{c,x,y,w_u,s,t,w_u'\} \cup \{w_{ua}',w_{st}\}$ is an isometric \WW, thus contradicting Theorem~\ref{th:charact}.

\bigskip
\noindent
\textit{\textbullet~Case $d(w_u,w_u') = d(w_v,w_v') = 1$.}
\medskip

\noindent
See Fig.~\ref{fig:forbidden-construct-2-case-2}.
Since there is a $C_4$ with vertex-set $\{x,w_u',w_u,y\}$ (respectively, $\{x,w_v',w_v,y\}$), by Lemma~\ref{lem:convex}, one of the edges $xw_u,yw_u'$ (respectively, $xw_v,yw_v'$) must be present. Furthermore, since by Theorem~\ref{th:charact} the disk $D(a,r_a+1)$ (respectively, $D(b,r_b+1)$) must be convex, one of the edges $w_u'y,w_v'y$ (respectively, $w_ux,w_vx$) must be missing. As a result, by symmetry, we only need to consider the subcase when $w_u'y,w_vx \in E$ but $w_ux,w_v'y \notin E$.
Then, $\{x,w_u',w_u,w_{ub},w_{vb},w_v\}$ induces a cycle $C_6$ and $Y = \{x,w_u',w_u,w_{ub},w_{vb},w_v\} \cup \{y\}$ induces a wheel $W_6$.
\begin{figure}[!ht]
    \centering
    \begin{tikzpicture}

            \draw (-1,0) -- (1,0);
            \draw (1,1) -- (3,1);
            \draw (1,1) -- (1,0); 
            \draw (1,0) -- (3,1);
            \draw (1,-1) -- (1,0);
            \draw (1,-1) -- (3,-1);
            \draw (1,0) -- (3,-1);
            \draw (-3,1) -- (-1,1) -- (-1,-1) -- (-3,-1);
            \draw (-3,1) -- (-1,0) -- (-3,-1);
            \draw (-3,-1) -- (-3,1);
            \draw (3,-1) -- (3,1);
            \draw (-1,1) -- (1,1);
            \draw (-1,-1) -- (1,-1);
            \draw (1,-1) -- (-1,0); \draw (-1,1) -- (1,0);
            \draw (1,1) -- (0,2) -- (-1,1);
            \draw (3,1) -- (4,0) -- (3,-1);

            \draw[thick,red] (0,2) -- (-1,1) -- (-1,0) -- (1,-1) -- (3,-1) -- (4,0) -- (3,1) -- (1,1) -- (0,2);
            \draw[thick,red] (-1,1) -- (1,1) -- (1,-1);
            \draw[thick,red] (-1,0) -- (1,0) -- (3,-1) -- (3,1) -- (1,0) -- (-1,1);

            \node[circle,fill=black,inner sep=0pt,minimum size=5pt, rotate=45,label=below:{$x$}] at (-1,0) {};
            \node[circle,fill=black,inner sep=0pt,minimum size=5pt, rotate=45, label=above:{$y$}] at (1,0) {};
            \node[circle,fill=black,inner sep=0pt,minimum size=5pt, rotate=-45,label=right:{$w_u$}] at (1,1) {};
            \node[circle,fill=black,inner sep=0pt,minimum size=5pt, label=above:{$w_{ub}$}] at (3,1) {};
            \node[circle,fill=black,inner sep=0pt,minimum size=5pt, rotate=-45,label=right:{$w_v$}] at (1,-1) {};
            \node[circle,fill=black,inner sep=0pt,minimum size=5pt, rotate=-45,label=right:{$w_{vb}$}] at (3,-1) {};
            \node[circle,fill=black,inner sep=0pt,minimum size=5pt, rotate=-45,label=below:{$w_u'$}] at (-1,1) {};
            \node[circle,fill=black,inner sep=0pt,minimum size=5pt, label=above:{$w_{ua}'$}] at (-3,1) {};
            \node[circle,fill=black,inner sep=0pt,minimum size=5pt, rotate=-45,label=right:{$w_v'$}] at (-1,-1) {};
            \node[circle,fill=black,inner sep=0pt,minimum size=5pt, rotate=-45,label=right:{$w_{va}'$}] at (-3,-1) {};
            \node[circle,fill=black,inner sep=0pt,minimum size=5pt, rotate=-45,label=above:{$\alpha$}] at (0,2) {};
            \node[circle,fill=black,inner sep=0pt,minimum size=5pt, rotate=-45,label=right:{$\beta$}] at (4,0) {};

    \end{tikzpicture}
    \caption{Case $d(w_u,w_u') = d(w_v,w_v')=1$.}
    \label{fig:forbidden-construct-2-case-2}
\end{figure}

\medskip
Let $\alpha \in N(w_u) \cap N(w_u') \cap S_{r_u-2}(u,v)$ and $\beta \in N(w_{ub}) \cap N(w_{vb}) \cap S_{r_b-1}(b,a)$, that both exist by Corollary~\ref{cor:c3}.
By the choice of $\alpha,\beta$, we have $N(\alpha) \cap Y = \{w_{u},w_{u}'\}$ and $N(\beta) \cap Y = \{w_{ub},w_{vb}\}$.
Therefore, $X = \{\alpha,\beta\} \cup Y$ induces a \WW.

Moreover,
\begin{itemize}
    \item $d(\alpha,\beta) \geq d(u,b) - d(u,\alpha) - d(b,\beta) = r_u + r_b - (r_u-2) - (r_b - 1) = 3$;
    \item $d(\alpha,w_{vb}) \geq d(u,w_{vb}) - d(u,\alpha) = r_u+1 - (r_u-2) = 3$;
    \item $d(\alpha,w_v) \geq d(u,v) - d(u,\alpha) - d(v,w_v) = r_u+r_v - (r_u-2) - (r_v-1) = 3$;
    \item $d(\beta,x) \geq d(a,b) - d(a,x) - d(b,\beta) = r_a+r_b+3-(r_a+1)-(r_b-1) = 3$;
    \item $d(\beta,w_u') \geq d(a,b) - d(a,w_u') - d(b,\beta) = r_a+r_b+3-(r_a+1)-(r_b-1) = 3$.
\end{itemize}
As a result, $G[X]$ is an isometric subgraph, thus contradicting Theorem~\ref{th:charact}.
\qed\end{proof}

We are now ready to prove the main result in this section.

\begin{theorem}\label{thm:alpha-1}
If $G=(V,E)$ is an $\alpha_1$-metric graph, then it is $1$-hyperbolic.
\end{theorem}
\begin{proof}
As we already discussed in the proof of Theorem~\ref{thm:main}, it is sufficient to prove that in ${\cal H}(G)$, for every $u,v \in V$ and for every $k \leq d(u,v)$, if $x,y \in S_{k}(u,v,{\cal H}(G))$ then we have $d(x,y) \leq 2$. For that, we need to consider three different cases.
\begin{itemize}
    \item {\it Case $x,y \in V$.} The result follows from Lemma~\ref{lem:thinness}.
    \item {\it Case $x \in V, \ y \notin V$} (the case $x \notin V, \ y \in V$ is symmetric to this one). For every $w \in V$, let $r_y(w) = d(y,w)$. By Lemma~\ref{lem:alpha-1-thinness}, we have $d(x,w) \leq r_y(w) + 2$. By Lemma~\ref{lem:dist-in-hull}, $d(x,y) \leq 2$.
    \item {\it Case $x,y \notin V$.} By Lemma~\ref{lem:max-sp}, there exist $x',y' \in V$ such that any shortest $(x,y)$-path in ${\cal H}(G)$ is contained in a shortest $(x',y')$-path in ${\cal H}(G)$. Let $r_x = d(x,x')$ and $r_y = d(y,y')$. By Lemma~\ref{lem:close-balls}, $d(x',y') \leq r_x + r_y +2$. Therefore, $d(x,y) \leq 2$.
\end{itemize}
The above proves, as claimed, that the interval thinness of ${\cal H}(G)$ is at most two, and therefore that $G$ is $1$-hyperbolic. \qed
\end{proof}

Next, we present two consequences of this result (in Section~\ref{sec:characterization-hyp} and Section~~\ref{sec:alg}, respectively).

    \subsection{A new characterization of $\alpha_1$-metric graphs}\label{sec:characterization-hyp}
    
The following characterization of 1/2-hyperbolic graphs was presented in \cite{BaCh2003}. 

\begin{theorem}[\cite{BaCh2003}]\label{lem:1/2-hyp}
 A graph $G$ is $1/2$-hyperbolic if and only if $G$ is $\alpha_1$-metric and none of the graphs of Fig. \ref{fig:1/2-hyp-forbid} occur as isometric subgraphs. 
\end{theorem}

  \begin{figure}[htb]
    \vspace{-10mm}
    \begin{center} 
          \begin{tikzpicture}
      \node[circle,fill=black,inner sep=0pt,minimum size=5pt, rotate=45] at (-1,.5) {};
      \node[circle,fill=black,inner sep=0pt,minimum size=5pt, rotate=45] at (0,.5) {};
      \node[circle,fill=black,inner sep=0pt,minimum size=5pt, rotate=45] at (1,.5) {};
      \node[circle,fill=black,inner sep=0pt,minimum size=5pt, rotate=45] at (-.5,0) {};
      \node[circle,fill=black,inner sep=0pt,minimum size=5pt, rotate=45] at (.5,0) {};
      \node[circle,fill=black,inner sep=0pt,minimum size=5pt, rotate=45] at (-1,-.5) {};
      \node[circle,fill=black,inner sep=0pt,minimum size=5pt, rotate=45] at (0,-.5) {};
      \node[circle,fill=black,inner sep=0pt,minimum size=5pt, rotate=45] at (1,-.5) {};
      \draw (-1,.5) -- (1,.5) -- (.5,0) -- (1,-.5) -- (-1,-.5) -- (-.5,0) -- (-1,.5);
      \draw (0,.5) -- (-.5,0) -- (0,-.5) -- (.5,0) -- (0,.5);
      \draw (-.5,0) -- (.5,0);
  \end{tikzpicture}\hfill
  \begin{tikzpicture}
      \node[circle,fill=black,inner sep=0pt,minimum size=5pt, rotate=45] at (-1,.5) {};
      \node[circle,fill=black,inner sep=0pt,minimum size=5pt, rotate=45] at (0,.5) {};
      \node[circle,fill=black,inner sep=0pt,minimum size=5pt, rotate=45] at (1,.5) {};
      \node[circle,fill=black,inner sep=0pt,minimum size=5pt, rotate=45] at (-.5,0) {};
      \node[circle,fill=black,inner sep=0pt,minimum size=5pt, rotate=45] at (.5,0) {};
      \node[circle,fill=black,inner sep=0pt,minimum size=5pt, rotate=45] at (-1,-.5) {};
      \node[circle,fill=black,inner sep=0pt,minimum size=5pt, rotate=45] at (0,-.5) {};
      \node[circle,fill=black,inner sep=0pt,minimum size=5pt, rotate=45] at (1,-.5) {};
      \draw (-1,.5) -- (1,.5) -- (.5,0) -- (1,-.5) -- (-1,-.5) -- (-.5,0) -- (-1,.5);
      \draw (0,.5) -- (-.5,0) -- (0,-.5) -- (.5,0) -- (0,.5);
      \draw (-.5,0) -- (.5,0); \draw (0,.5) -- (0,-.5);
  \end{tikzpicture}\hfill
  \begin{tikzpicture}
      \node[circle,fill=black,inner sep=0pt,minimum size=5pt, rotate=45] at (0,1) {}; 
      \node[circle,fill=black,inner sep=0pt,minimum size=5pt, rotate=45] at (-.5,.5) {};
      \node[circle,fill=black,inner sep=0pt,minimum size=5pt, rotate=45] at (.5,.5) {};
      \node[circle,fill=black,inner sep=0pt,minimum size=5pt, rotate=45] at (-1,0) {};
      \node[circle,fill=black,inner sep=0pt,minimum size=5pt, rotate=45] at (0,0) {};
      \node[circle,fill=black,inner sep=0pt,minimum size=5pt, rotate=45] at (1,0) {};
      \node[circle,fill=black,inner sep=0pt,minimum size=5pt, rotate=45] at (0,-1) {}; 
      \node[circle,fill=black,inner sep=0pt,minimum size=5pt, rotate=45] at (-.5,-.5) {};
      \node[circle,fill=black,inner sep=0pt,minimum size=5pt, rotate=45] at (.5,-.5) {};
      \draw (0,1) -- (-.5,.5) -- (-1,0) -- (-.5,-.5) -- (0,-1) -- (.5,-.5) -- (1,0) -- (.5,.5) -- (0,1);
      \draw (-1,0) -- (1,0); \draw (-.5,.5) -- (.5,.5); \draw (-.5,-.5) -- (.5,-.5);
      \draw (-.5,.5) -- (.5,-.5); \draw (.5,.5) -- (-.5,-.5);
  \end{tikzpicture}\hfill
  \begin{tikzpicture}
      \node[circle,fill=black,inner sep=0pt,minimum size=5pt, rotate=45] at (0,1) {};
      \node[circle,fill=black,inner sep=0pt,minimum size=5pt, rotate=45] at (-.75,.75) {};
      \node[circle,fill=black,inner sep=0pt,minimum size=5pt, rotate=45] at (.75,.75) {};
      \node[circle,fill=black,inner sep=0pt,minimum size=5pt, rotate=45] at (-1.25,.25) {};
      \node[circle,fill=black,inner sep=0pt,minimum size=5pt, rotate=45] at (1.25,.25) {};
      \node[circle,fill=black,inner sep=0pt,minimum size=5pt, rotate=45] at (0,0) {};
      \node[circle,fill=black,inner sep=0pt,minimum size=5pt, rotate=45] at (-.75,-.25) {};
      \node[circle,fill=black,inner sep=0pt,minimum size=5pt, rotate=45] at (.75,-.25) {};
      \node[circle,fill=black,inner sep=0pt,minimum size=5pt, rotate=45] at (-.5,-.75) {};
      \node[circle,fill=black,inner sep=0pt,minimum size=5pt, rotate=45] at (.5,-.75) {};
      \node[circle,fill=black,inner sep=0pt,minimum size=5pt, rotate=45] at (0,-1.25) {};
      \draw (0,1) -- (-.75,.75) -- (-1.25,.25) -- (-.75,-.25) -- (-.5,-.75) -- (0,-1.25) -- (.5,-.75) -- (.75,-.25) -- (1.25,.25) -- (.75,.75) -- (0,1); 
      \draw (0,1) -- (0,0) -- (-.75,.75) -- (-.75,-.25) -- (0,0) -- (-.5,-.75) -- (.5,-.75) -- (0,0) -- (.75,-.25) -- (.75,.75) -- (0,0);
  \end{tikzpicture}\hfill
  \begin{tikzpicture}
     \node[circle,fill=black,inner sep=0pt,minimum size=5pt, rotate=45] at (0,1.5) {};
     \node[circle,fill=black,inner sep=0pt,minimum size=5pt, rotate=45] at (-.5,1) {};
     \node[circle,fill=black,inner sep=0pt,minimum size=5pt, rotate=45] at (.5,1) {};
     \node[circle,fill=black,inner sep=0pt,minimum size=5pt, rotate=45] at (-.75,.5) {};
     \node[circle,fill=black,inner sep=0pt,minimum size=5pt, rotate=45] at (.75,.5) {};
     \node[circle,fill=black,inner sep=0pt,minimum size=5pt, rotate=45] at (-1.25,0) {};
     \node[circle,fill=black,inner sep=0pt,minimum size=5pt, rotate=45] at (0,0) {};
     \node[circle,fill=black,inner sep=0pt,minimum size=5pt, rotate=45] at (1.25,0) {};
     \node[circle,fill=black,inner sep=0pt,minimum size=5pt, rotate=45] at (0,-1.5) {};
     \node[circle,fill=black,inner sep=0pt,minimum size=5pt, rotate=45] at (-.5,-1) {};
     \node[circle,fill=black,inner sep=0pt,minimum size=5pt, rotate=45] at (.5,-1) {};
     \node[circle,fill=black,inner sep=0pt,minimum size=5pt, rotate=45] at (-.75,-.5) {};
     \node[circle,fill=black,inner sep=0pt,minimum size=5pt, rotate=45] at (.75,-.5) {};
     \draw (0,1.5) -- (-.5,1) -- (-.75,.5) -- (-1.25,0) -- (-.75,-.5) -- (-.5,-1) -- (0,-1.5) -- (.5,-1) -- (.75,-.5) -- (1.25,0) -- (.75,.5) -- (.5,1) -- (0,1.5);
     \draw (0,0) -- (.5,1) -- (-.5,1) -- (0,0) -- (-.75,.5) -- (-.75,-.5) -- (0,0) -- (-.5,-1) -- (.5,-1) -- (0,0) -- (.75,-.5) -- (.75,.5) -- (0,0);
  \end{tikzpicture}
        \caption{\label{fig:1/2-hyp-forbid} Forbidden isometric subgraphs for 1/2-hyperbolic graphs.} %
    \end{center}
   \vspace*{-8mm}
  \end{figure}

Thus, it follows  from Theorem \ref{thm:alpha-1} and Theorem \ref{lem:1/2-hyp} that the class of $\alpha_1$-metric graphs is a superclass of 1/2-hyperbolic graphs and a subclass of 1-hyperbolic graphs. Theorem \ref{lem:1/2-hyp} describes the place of 1/2-hyperbolic graphs within the class of $\alpha_1$-metric graphs in terms of forbidden isometric subgraphs.   
In what follows, we can also describe the place of $\alpha_1$-metric graphs within the class of 1-hyperbolic graphs in terms of forbidden isometric subgraphs. 
Our work in this sub-section is inspired by the following recent characterization of graphs with convex disks.

\begin{theorem}[see Corollary 4.21 in~\cite{chalopin2022graphs-M}, v1]\label{tm:CB-graphs}
All disks of a graph $G$ are convex if and only if $G$ has no isometric $C_k$ for every $k\ge 4$ such that $k\neq 5$, no isometric $PT$, every induced subgraph isomorphic to $PP_1$ has diameter at most 3 and every induced subgraph isomorphic to $PP_2$ has diameter 2 (see Fig. \ref{fig:CB-forbid}).
\end{theorem}

  \begin{figure}[htb]
   \begin{center} 
        \begin{subfigure}{.33\textwidth}
            \centering
            \begin{tikzpicture}
                \node[circle,fill=black,inner sep=0pt,minimum size=5pt, rotate=45] at (-1,0) {};
                \node[circle,fill=black,inner sep=0pt,minimum size=5pt, rotate=45] at (1,0) {};]
                \node[circle,fill=black,inner sep=0pt,minimum size=5pt, rotate=45] at (-.5,.5) {};
                \node[circle,fill=black,inner sep=0pt,minimum size=5pt, rotate=45] at (.5,.5) {};
                \node[circle,fill=black,inner sep=0pt,minimum size=5pt, rotate=45] at (-.5,-.5) {};
                \node[circle,fill=black,inner sep=0pt,minimum size=5pt, rotate=45] at (.5,-.5) {};
                \draw (-1,0) -- (-.5,.5) -- (.5,.5) -- (1,0) -- (.5,-.5) -- (-.5,-.5) -- (-1,0); \draw (.5,.5) -- (.5,-.5);
            \end{tikzpicture}
            \caption{$PT$}
            \label{fig:PT}
        \end{subfigure}\hfill
        \begin{subfigure}{.33\textwidth}
            \centering
            \begin{tikzpicture}
                \node[circle,fill=black,inner sep=0pt,minimum size=5pt, rotate=45] at (0,.25) {};
                \node[circle,fill=black,inner sep=0pt,minimum size=5pt, rotate=45] at (0,-.25) {};
                \node[circle,fill=black,inner sep=0pt,minimum size=5pt, rotate=45] at (-1,0) {};
                \node[circle,fill=black,inner sep=0pt,minimum size=5pt, rotate=45] at (1,0) {};
                \node[circle,fill=black,inner sep=0pt,minimum size=5pt, rotate=45] at (-.75,.5) {};
                \node[circle,fill=black,inner sep=0pt,minimum size=5pt, rotate=45] at (-.75,-.5) {};
                \node[circle,fill=black,inner sep=0pt,minimum size=5pt, rotate=45] at (.75,.5) {};
                \node[circle,fill=black,inner sep=0pt,minimum size=5pt, rotate=45] at (.75,-.5) {};
                \draw (-1,0) -- (-.75,.5) -- (0,.25) -- (.75,.5) -- (1,0) -- (.75,-.5) -- (0,-.25) -- (-.75,-.5) -- (-1,0); \draw (0,.25) -- (0,-.25);
            \end{tikzpicture}
            \caption{$PP_1$}
            \label{fig:PP1}
        \end{subfigure}\hfill
        \begin{subfigure}{.33\textwidth}
            \centering
            \begin{tikzpicture}
                \node[circle,fill=black,inner sep=0pt,minimum size=5pt, rotate=45] at (-1,0) {};
                \node[circle,fill=black,inner sep=0pt,minimum size=5pt, rotate=45] at (1,0) {};
                \node[circle,fill=black,inner sep=0pt,minimum size=5pt, rotate=45] at (-.75,.5) {};
                \node[circle,fill=black,inner sep=0pt,minimum size=5pt, rotate=45] at (-.75,-.5) {};
                \node[circle,fill=black,inner sep=0pt,minimum size=5pt, rotate=45] at (.75,.5) {};
                \node[circle,fill=black,inner sep=0pt,minimum size=5pt, rotate=45] at (.75,-.5) {};
                \node[circle,fill=black,inner sep=0pt,minimum size=5pt, rotate=45] at (0,0) {};
                \draw (-1,0) -- (-.75,.5) -- (.75,.5) -- (1,0) -- (.75,-.5) -- (-.75,-.5) -- (-1,0) -- (1,0);
            \end{tikzpicture}
            \caption{$PP_2$}
            \label{fig:PP2}
        \end{subfigure}
        \caption{\label{fig:CB-forbid} Some obstructions to disk convexity.} %
    \end{center}
   \vspace*{-10mm}
  \end{figure}

The following result would follow from Theorem \ref{th:charact}, Theorem \ref{thm:alpha-1},  Theorem \ref{tm:CB-graphs}, and the fact that a $1$-hyperbolic graph cannot contain any isometric $C_k$ for every $k\ge 8$.
\begin{corollary} \label{cor:alpha_1-to-hyp}
 A graph $G$ is  $\alpha_1$-metric if and only if $G$ is $1$-hyperbolic, has no isometric  $C_4$, $C_6$, $C_7$,  \WW\ or $PT$, every induced subgraph isomorphic to ${PP}_1$ has diameter at most 3 and every induced subgraph isomorphic to $PP_2$ has diameter 2 (see Fig. \ref{fig:forbid} and Fig. \ref{fig:CB-forbid}).  
\end{corollary}

However, it was brought to our attention that Theorem~\ref{tm:CB-graphs} was not included in the journal version of~\cite{chalopin2022graphs-M}, see~\cite{chalopin2022graphs}.
So, we decided to give a direct proof of our Corollary~\ref{cor:alpha_1-to-hyp}.
\begin{proof}[Corollary~\ref{cor:alpha_1-to-hyp}]
We prove both implications separately.

$(\Longrightarrow)$ Let $G$ be an $\alpha_1$-metric graph. By Theorem \ref{thm:alpha-1}, $G$ is $1$-hyperbolic. By Theorem \ref{th:charact}, $G$ cannot have an isometric \WW, and all disks of $G$ are convex. Hence, by Lemma \ref{lem:convex}, 
$G$ does not contain isometric cycles of length $\ell > 5$, and for any two nonadjacent vertices $x,y$ of $G$ the slices $S_1(x,y)$ and $S_1(y,x)$ are cliques. The latter implies that $G$ cannot have an isometric $C_4$. It also implies that $G$ cannot have an isometric $PT$, see Fig.~\ref{fig:PT}, because for the two vertices at distance 3 in PT, one of the corresponding slices is not a clique. 
In the same way, $G$ cannot have an induced $PP_1$ with diameter 4, see Fig.~\ref{fig:PP1}, which is because for the vertices at distance 4 in $PP_1$, the two corresponding slices are not cliques. 
Finally, $G$ cannot have an induced $PP_2$ with diameter 3, see Fig.~\ref{fig:PP2}, because for any two vertices at distance 3 in $PP_2$, the corresponding slices are not cliques. \medskip

$(\Longleftarrow)$ 
In what follows, we only prove that all the conditions implying convexity of disks in $G$ hold.  
Since we forbid \WW as an isometric subgraph, by Theorem~\ref{th:charact}, this will prove that $G$ is $\alpha_1$-metric.
Like in \cite{chalopin2022graphs}, we consider the following two conditions, defined for every $k \in \mathbb{N}$. 
\begin{itemize}
\item {\em Interval Neighbourhood condition} (INC$_k$):  for any two vertices $u,v$ with $d(u,v)=k+1$, the slice $S_1(u,v)$ is a clique. 
\item {\em Triangle-Pentagon condition} (TPC$_k$):   for any three vertices $v,x,y$ such that $d(x,v)=d(y,v) =k+1$ and $x\sim y$, either there exists a vertex $z$ with $d(z,v)=k$ such that the vertices $x,y,z$ induce a $C_3$ in $G$, or there exist vertices $z,u,w$ such that $d(z,v)=k-1$ and the vertices $x,y,w,z,u$ induce a $C_5$ in $G$. 
\end{itemize}

We will prove, by induction on $k$, that both those conditions are true for every $k\ge 0$. 
It is straightforward to check that they are always true for $k=0$ (for any graph $G$).
Furthermore, since we forbid $C_4$ as an induced subgraph, the conditions are true for $k=1$.
Next, we check those conditions separately for $k=2$. 
\begin{itemize}
    \item {\it Case $k=2$.}  
    
   ~~~~~To show INC$_2$, let us consider arbitrary vertices $u,v$ with $d(u,v)=3$ and any two distinct $x,y \in S_1(u,v)$. 
   Suppose, for the sake of contradiction, that $x$ and $y$ are not adjacent. 
   We also consider some arbitrary $x' \in N(x) \cap N(v)$, and $y'\in N(y) \cap N(v)$. 
   Since $G$ has no induced $C_4$, necessarily $x'\ne y'$, $y\nsim x'$ and  $x\nsim y'$. 
   Consequently, there is a (not necessarily induced) cycle of length six going through the vertices $u,x,x',v,y',y,u$. 
   Furthermore, the only possible chord of this cycle is $x'y'$.
   However, if $x'\sim y'$, then as $d(u,v)=3$, we get an isometric $PT$, which is forbidden.   
   Therefore, this cycle must be induced, but due to our assumptions it is not isometric.
   It implies that either $d(x,y') = 2$ or $d(x',y) = 2$.
   Without loss of generality, let us assume $d(x,y')=2$ and let $w$ be a common neighbour of $x$ and $y'$. 
   As $d(u,v)=3$, and as we assume that every induced $PP_2$ of $G$ has diameter two, the vertices $u,y,y',v,x',x$ and $w$ cannot induce a $PP_2$. 
   Therefore, $w$ must have some neighbour within $x',y,u,v$.  
   In fact, since $d(u,v)=3$ and $G$ has no induced $C_4$, $w$ is either adjacent to both $u$ and $y$ or it is adjacent to both $x'$ and $v$. 
   By symmetry, let us assume $w\sim u$ and $w\sim y$.  
   But then, the vertices $u,x,w,x',y',v$ form an isometric $PT$, which is forbidden.
   Therefore, necessarily, $x\sim y$.

    ~~~~~To show TPC$_2$, let us now consider three arbitrary vertices $v,x,y$ such that $d(x,v)=d(y,v) =3$ and $x\sim y$. 
    We pick some vertices $x'\in S_1(x,v)$ and $y'\in S_1(y,v)$ minimizing $d(x',y')$. 
    If $x'=y'$, then we are done as $x,y,x'$ induce a $C_3$ in $G$ and $d(x',v)=2$. 
    So, we may assume from now on that $x'\nsim y$, $y'\nsim x$.
    Furthermore, as $G$ cannot have any induced $C_4$, $x' \not\sim y'$.
    Hence, $d(x',y')\ge 2$. 
    Let us consider some arbitrary vertices $u \in N(x') \cap N(v)$, and $w \in N(y') \cap N(v)$.
    If $d(u,y)=2$ (or $d(w,x)=2$) then as, by the induction hypothesis, TPC$_1$ holds for $u,x,y$ (or $w,x,y$), TPC$_2$ holds for $v,x,y$.
    So, we can assume $d(u,y)=3$ and $d(w,x)=3$.
    Furthermore, as $x\nsim y'$ and INC$_2$ holds for $u,y$, we cannot have $d(u,y') = 2$. 
    Hence, $d(u,y')=3$ and, by a symmetric reasoning, $d(w,x')=3$. 
    As $x',v \in N(u)$, $x' \nsim v$ and INC$_2$ holds for $u,y'$, we cannot have $d(x',y') = 2$.
    It implies $d(x',y') = 3$.
    But then, the vertices $x,y,y',w,v,u,x'$ form an isometric $C_7$, which is forbidden.
    Therefore, TPC$_2$ holds.  \medskip
    
  \item {\it Case $k>2$.} 
  
   ~~~~~To show INC$_k$, $k\ge 3$, let us consider arbitrary vertices $u,v$ with $d(u,v)=k+1\ge 4$ and any two distinct $x,y \in S_1(u,v)$. 
   Suppose, for the sake of contradiction, that $x$ and $y$ are not adjacent. 
   We pick some vertices $x' \in S_1(x,v)$ and $y'\in S_1(y,v)$ minimizing $d(x',y')$. 
   Since we assume $G$ to be $1$-hyperbolic, it has interval thinness at most 2.
   In particular, as $x',y' \in S_2(u,v)$, $d(x',y')\le 2$.
   Furthermore, as $G$ has no induced $C_4$, necessarily, $x'\ne y'$, $y\nsim x'$ and  $x\nsim y'$. 
   It implies that either $x' \sim y'$ or $d(x',y') = 2$.
   If $x' \sim y'$ then, by the induction hypothesis, TPC$_{k-2}$ holds for $v,x',y'$.
   Therefore, we either get a vertex $z$ with $d(z,v)=k-2$ such that $x',y',z$ induce a $C_3$, or three vertices $z,a,b$ such that $d(z,v)=k-3$ and the vertices $x',y',b,z,a$ induce a $C_5$.
  In the first situation, the vertices $u,x,y,x',y',z$ form an isometric $PT$, which is forbidden.
  In the second situation, the vertices $u,x,y,x',y',a,b,z$ form an induced $PP_1$ with  $d(z,u)=4$, which is also forbidden. 
  So, we can assume: \medskip
   
   \begin{itemize}
    \item[(i) ] $d(x',y')=2$ for every $x'\in S_1(x,v)$ and every $y' \in S_1(y,v)$; 
    \item[(ii)] moreover (by the contrapositive of (i)), for any distinct $a,b \in S_1(u,v)$, if there exists some $a'\in S_1(a,v)$ and some $b'\in S_1(b,v)$ with $a' \sim b'$, then $a \sim b$.
   \end{itemize} \medskip
   
Let $t \in N(x') \cap N(y')$ be arbitrary. 
If $d(t,u)=3$, then $x,y \in S_1(u,t)$ and as, by the induction hypothesis, INC$_2$ holds for $u,t$, we obtain $x \sim y$, a contradiction.
Therefore, $d(t,u) \le 2$.
It implies $d(v,t) \ge d(u,v)-d(t,u) \ge k+1-2 = k-1$.
Furthermore, $d(t,v) \le 1 + d(x',v) = k$.
If $d(t,v)=k$, then $x',y' \in S_1(t,v)$ and as, by the induction hypothesis, INC$_{k-1}$ holds for $v,t$, we obtain $x' \sim y'$, thus contradicting (i). 
Consequently, $d(t,v) = k-1$, and so $t \in S_2(u,v)$.
Since $t$ was chosen arbitrarily: \medskip

    \begin{itemize}
        \item[(iii)] $N(x') \cap N(y') \subseteq S_2(u,v)$.
    \end{itemize}\medskip

Let $t' \in N(u) \cap N(t)$ be arbitrary.
Since $t \sim x'$, by (i), we have $t \nsim y$.
Similarly, since $t \sim y'$, by (i), we have $t \nsim x$.
Therefore, $t' \ne x$, and $t' \ne y$.
By (ii) applied for $a = x, b = t'$ and $a'= x', b' = t$, necessarily $t'\sim x$.
Similarly, by (ii) applied for $a = y, b = t'$ and $a' = y', b' = t$, necessarily $t' \sim y$.
If $t' \nsim x'$, then $x,x',t,t'$ induce a $C_4$, which is forbidden.
Therefore, $t' \sim x'$ and, by a symmetric reasoning, $t' \sim y'$.
However, as $t' \sim u$, the latter contradicts (iii).
Thus, INC$_k$, $k\ge 3$, holds.


  \medskip
  \noindent
~~~~~To show TPC$_k$, $k\ge 3$, the proof is similar to what we did for TPC$_2$.
More precisely, let us consider three arbitrary vertices $v,x,y$ such that $d(x,v)=d(y,v) =k+1$ and $x\sim y$. 
We pick some vertices $x'\in S_1(x,v)$ and $y'\in S_1(y,v)$ minimizing $d(x',y')$.
If $x'=y'$, then we are done as $x,y,x'$ induce a $C_3$ in $G$ and $d(x',v)=k$. 
So, we may assume from now on that $x'\nsim y$, $y'\nsim x$.
Furthermore, as $G$ cannot have any induced $C_4$, $x' \not\sim y'$.
Hence, $d(x',y')\ge 2$.
Let us consider some arbitrary vertices $u \in S_1(x',v)$ and $w \in S_1(y',v)$.
If $d(u,y)=2$ (or $d(w,x)=2$) then as, by the induction hypothesis, TPC$_1$ holds for $u,x,y$ (or $w,x,y$), TPC$_k$ holds for $v,x,y$.
So, we can assume $d(u,y)=3$ and $d(w,x)=3$ for every $u \in S_1(x',v)$ and every $w \in S_1(y',v)$. 
Furthermore, as $x\nsim y'$ and, by the induction hypothesis, INC$_2$ holds for $u,y$, we cannot have $d(u,y') = 2$. 
Hence, $d(u,y') \ge 3$ for every $u \in S_1(x',v)$ and, by a symmetric reasoning, $d(w,x') \ge 3$ for every $w \in S_1(y',v)$. 

~~~~~Assume $d(x',y')=2$, and let $t\in N(x') \cap N(y')$ be arbitrary. 
Since $t \sim x'$, and $d(v,x') = k$, we get $k-1 \le d(v,t) \le k+1$.
If $d(v,t)=k-1$, then we are done as $t,x',x,y,y'$ induce a $C_5$ with $d(t,v)=k-1$.   
Furthermore, if $d(v,t)=k+1$, then $x',y' \in S_1(t,v)$, and as INC$_k$ holds for $v,t$, we obtain $x' \sim y'$, a contradiction. 
So, we may assume from now on that $d(v,t) = k$ for every $t \in N(x') \cap N(y')$. 
If $x \sim t$, then $x,t,y',y$ form a cycle, and as $G$ has no induced $C_4$, necessarily $y \sim t$. Then, we are done, as $t,x,y$ induce a $C_3$ with $d(t,v)=k$.
Therefore, we can also assume from now on that $t \nsim x$, and by a symmetric argument, $t \nsim y$.
As, by the induction hypothesis, TPC$_{k-1}$ holds for $v,x',t$, we either get a vertex $z$ such that $d(z,v) = k-1$ and $x',t,z$ induce a $C_3$ or three vertices $a,b,z$ such that $d(z,v) = k-2$ and $x',t,a,z,b$ induce a $C_5$.
In the first situation, as $z \in S_1(x',v)$, $d(z,y)=3$. In particular, the vertices $z,x',t,x,y',y$ form an isometric $PT$, which is forbidden.  
In the second situation, $d(z,y) \le d(z,x')+d(x',y) = 4$. Furthermore, $d(z,y) \ge d(y,v)-d(z,v) = 3$.
Consequently, either $d(z,y) = 3$ or $d(z,y) = 4$.
If $d(z,y)=4$, then the vertices $z,a,b,t,x',y',x,y$ induce a $PP_1$ with diameter 4, which is also forbidden. 
Otherwise, as $d(z,x) = d(z,y)=3$, by the induction hypothesis, TPC$_2$ holds for $z,x,y$.
Since $z \in I(x,v) \cap I(y,v)$, the latter implies TPC$_k$ for $v,x,y$.

~~~~~From now on, we assume $d(x',y')=3$ for every $x'\in S_1(x,v)$ and every $y'\in S_1(y,v)$. 
Recall that $u \in S_1(x',v)$ and $w \in S_1(y',v)$.
We now consider the distance between $u$ and $w$.
Since $d(u,y') \ge 3$, $u \nsim w$.
If $d(u,w)=2$, then let $t \in N(u) \cap N(w)$ be arbitrary.
The vertices $u,t,w,y',y,x,x'$ form a (not necessarily induced) cycle of length seven.
Furthermore, any chord of this cycle must be an edge incident to $t$.
Since $d(u,y) = 3$, $t \nsim y$. In the same way: since $d(w,x) = 3$, $t \nsim x$; since $d(u,y') \ge 3$, $t \nsim y'$; and since $d(w,x') \ge 3$, $t \nsim x'$.
Consequently, the vertices $u,t,w,y',y,x,x'$ induce a $C_7$. However, due to our assumptions, this cycle is not isometric.
It implies that either $d(t,x) = 2$ or $d(t,y) = 2$.
Without loss of generality, let us assume $d(t,x) = 2$.
Then necessarily, $d(v,t)\ge k-1$.
Furthermore, $d(v,t) \le d(v,u)+1=k$.
Therefore, either $d(v,t) = k-1$ or $d(v,t) = k$.
If $d(v,t)=k$  then, as $u,w \in S_1(t,v)$ and, by the induction hypothesis, INC$_{k-1}$ holds for $v,t$, we obtain $u \sim w$, a contradiction.
Hence, $d(v,t) = k-1$.
Moreover, if $d(y,t)=2$, then as we also have $d(x,t) = 2$, by the induction hypothesis, TPC$_1$ holds for $t,x,y$.
Since in this situation $t \in I(x,v) \cap I(y,v)$, the latter implies that TPC$_k$ holds for $v,x,y$.
Hence, we may assume from now on $d(y,t)=3$.
Let $t' \in N(x) \cap N(t)$ be arbitrary.
As $t'\in S_1(x,v)$ and $y'\in S_1(y,v)$, $d(t',y')=3$.
But then, the vertices $x,y,y',w,t,t'$ form an isometric $C_6$ in $G$, which is forbidden.
Summarizing, either $d(u,w) > 2$ or a contradiction arises.
Now, for the $4$-tuple $x,u,w,v$, we consider the following three distance sums: (A) $d(x,v)+d(u,w)=k+1+d(u,w)>k+3$; (B)  $d(x,u)+d(v,w)=k+1$; (C)  $d(x,w)+d(v,u)=k+2$. 
As $G$ is 1-hyperbolic, the difference between two largest sums must be at most 2. 
Consequently, $2\ge $(A)-(C)$=d(x,v)+d(u,w)-d(x,w)-d(v,u)=k+1+d(u,w)-(k+2)=d(u,w)-1$, i.e., $d(u,w)\le 3$. 
Altogether combined, we conclude that $d(u,w)= 3$, and this holds for every $u\in S_1(x',v)$ and every $w\in S_1(y',v)$. 

~~~~~As we have established above, we may assume from now on: \medskip

   \begin{itemize}
    \item[(j) ] $d(x',y')=3$ for every $x'\in S_1(x,v)$ and every $y'\in S_1(y,v)$; 
    \item[(jj)] moreover, for every $u\in S_1(x',v)$ and every $w\in S_1(y',v)$, \\ $d(u,y)= d(w,x)= d(u,w)= 3$. 
   \end{itemize} 
\medskip

Let us also recall that $d(u,y') \ge 3$, and $d(w,x') \ge 3$.
Suppose, for the sake of contradiction, $d(u,y') = 4$.
As $k \ge 3$, by the induction hypothesis, INC$_3$ holds for $u,y'$.
In particular, as $w,y \in S_1(y',u)$, we obtain $w \sim y$, a contradiction.
By a symmetric reasoning, if we suppose $d(w,x') = 4$, then a contradiction arises.
Hence, \medskip 

  \begin{itemize}
      \item[(jjj)] $d(u,y') = d(w,x') = 3$, for every $u\in S_1(x',v)$ and every $w\in S_1(y',v)$.
  \end{itemize}
  \medskip

We now consider a shortest path $(u,a,b,y')$ between $u$ and $y'$. 
Since $b \sim y'$, and $d(v,y') = k$, we get $k-1 \le d(v,b) \le k+1$.
In the same way, since $a \sim u$ and $d(v,u) = k-1$, we get $k-2 \le d(v,a) \le k$.
If $d(v,b)=k+1$, then necessarily $d(v,a) = k$.
In this situation, $a,y' \in S_1(b,v)$.
As INC$_k$ holds for $b,v$, we obtain $a \sim y'$, a contradiction.
Consequently, $d(v,b) \le k$.
Furthermore, if  $d(v,b)=k-1$, then $b \in S_1(y',v)$ and so, by (jj) applied for $w=b$, we get $d(u,b)=3$, which is impossible.
Hence, $d(b,v)=k$ must hold.
Note that $x,y,y',b,a,u,x'$ form a cycle of length seven.
Furthermore, any chord of this cycle must be an edge incident to either $a$ or $b$.
As $(u,a,b,y')$ is a shortest path, $a \nsim y'$.
As $d(u,y) = 3$, we also get $a \nsim y$.
If $a\sim x$, then necessarily $d(v,a) = k$, and $a \in S_1(x,v)$.
As $y' \in S_1(y,v)$, by (j) applied for $x'=a$, we get $d(a,y') = 3$, a contradiction.
Similarly, if $a \sim x'$, then in this situation $a,x \in S_1(x',y')$.
As, by the induction hypothesis, INC$_2$ holds for $x',y'$, we obtain $a \sim x$, and a contradiction arises.
Hence, there is no chord incident to $a$.
Next, we prove that there is no chord incident to $b$ either.
Clearly, $b\nsim u$.
As there is no induced $C_4$ in $G$, $b \nsim x'$.
If $b \sim y$, then in this situation $b \in S_1(y,v)$, and so, by (jjj) applied for $y'=b$, $d(u,b) = 3$, a contradiction.
Finally, if $b \nsim y$ but $b \sim x$, then the vertices $b,x,y,y'$ induce a $C_4$, which is forbidden.
Thus, the vertices $x,y,y',b,a,u,x'$ induce a $C_7$, but due to our assumptions this cycle is not isometric. 

~~~~~In this situation, two vertices at distance $3$ in the cycle must be at distance $2$ in $G$.
Recall that by (j), $d(x',y') = 3$; by (jj), $d(u,y) = 3$; and by (jjj), $d(u,y') = 3$.
Moreover, if $d(b,x') = 2$, then necessarily $b,y \in S_1(y',x')$. As, by the induction hypothesis, INC$_2$ holds for $x',y'$, we obtain $b \sim y$, a contradiction.
Therefore, we only need to check the distances $d(b,x)$, $d(a,x)$ and $d(a,y)$.
If $d(a,y)=2$, then necessarily $a,x' \in S_1(u,y)$. As, by the induction hypothesis, INC$_2$ holds for $u,y$, we obtain $a \sim x'$, a contradiction.
In the same way, if $d(b,x)=2$, then as $d(b,x') = 3$, necessarily $u,x \in S_1(x',b)$. As, by the induction hypothesis, INC$_2$ holds for $x',b$, we obtain $u \sim x$, a contradiction.
Consequently, $d(a,x) = 2$.
As $d(a,y) = 3$, we get $x,y' \in S_1(y,a)$.
However, as by the induction hypothesis, INC$_2$ holds for $a,y$, we obtain $x \sim y'$, a contradiction.

~~~~~Overall, as $G$ cannot have any isometric $C_7$,  TPC$_k$, $k\ge 3$, must hold.  
\end{itemize}

Now, Corollary  \ref{cor:alpha_1-to-hyp} follows from  INC$_k$, $k\ge 0$,  Lemma \ref{lem:convex} and Theorem \ref{th:charact}. \qed 
\end{proof}

    \subsection{Algorithmic Applications}\label{sec:alg}
    
Recall that the \emph{eccentricity}~$e(v)$ of a vertex~$v$ in $G$ is defined by $\max_{u \in V} d(u, v)$, i.e., it is the distance to a most distant vertex. 
The {\em diameter} of a graph is the maximum over the eccentricities of all vertices: $diam(G) = \max_{u\in V} e(u)=\max_{u,v\in V} d(u,v)$.
It was left as an open question in  \cite{DrDu23} whether in $\alpha_1$-metric graphs the eccentricity of any  vertex furthest from an arbitrary vertex is at least $diam(G)-2$. Our Theorem \ref{thm:alpha-1} provides an affirmative answer to that question. Indeed, it is known that in any $\delta$-hyperbolic graph the eccentricity of any vertex furthest from an arbitrary vertex is at least $diam(G)-2\delta$ \cite{ChDEHV08}. 
So, we have the following corollary from Theorem \ref{thm:alpha-1}. 

\begin{corollary} \label{cor:alpha_1-d-2}
In every $\alpha_1$-metric  graph $G$,  the eccentricity of any  vertex furthest from an arbitrary vertex is at least $diam(G)-2$. Hence, an additive 2-approximation of the diameter of $G$ can be computed in linear time by a Breadth-First-Search. 
\end{corollary}

Note that computing the exact diameter in subquadratic time, even in  chordal graphs (a proper subclass of $\alpha_1$-metric  graphs), is impossible unless the well known Strong Exponential Time Hypothesis (SETH) is false \cite{BoCrHa16}.

\section{Conclusion}\label{sec:ccl}

Our main result in the paper is that every $\alpha_i$-metric graph must be $f(i)$-hyperbolic, for some function $f$ such that $\frac{i+1}{2} \le f(i) \le 3 \cdot \frac{i+1}{2}$.
It would be interesting to close the gap between the upper bound and the lower bound.
We only managed to do so for $\alpha_1$-metric graphs.
By contrast, general $\alpha_i$-metric graphs ($i \in \mathbb{N}$) seem much less structured.
For instance, they can have metric triangles with side-length unbounded (Lemma~\ref{lem:unbounded-metric-triangle}).
Furthermore, unlike $\delta$-hyperbolic graphs, the $\alpha_i$-metric graphs are not well-behaved under important operations such as $1$-subdivision (see Fig.~\ref{fig:unbounded-alpha-i}) and injective hull (see Fig.~\ref{fig-thinness-sharp}).

Nevertheless, there is  
a natural generalization of an $\alpha_i$-metric, which we call a {\em ($\lambda,\mu$)-bow metric}:  namely, if two shortest paths $P(u,w)$ and $P(v,x)$ share a common shortest subpath $P(v,w)$ of length more than $\lambda$ (that is, they overlap by more than  $\lambda$),  then the distance between $u$ and $x$ is at least $d(u,v)+d(v,w)+d(w,x)-\mu$. Clearly, $\alpha_i$-metric graphs are ($0,i$)-bow metric. However, this generalization is more robust to some graph operations. For instance, the $1$-subdivision of an ($\lambda,\mu$)-bow metric graph must be ($2\lambda+2,2\mu+2$)-bow metric. This notion of ($\lambda,\mu$)-bow metric can also be considered for all geodesic metric spaces. Furthermore, in what follows, we show that every $\delta$-hyperbolic graph (in fact, every  $\delta$-hyperbolic geodesic metric space) is ($\delta, 2\delta$)-bow metric.

\begin{proposition} \label{prop:bow-metric}
Every $\delta$-hyperbolic graph and, generally, every   $\delta$-hyperbolic geodesic metric space is ($\delta, 2\delta$)-bow metric. 
\end{proposition}
 
\begin{proof} Consider four arbitrary vertices $u,v,w,x$ such that $v\in I(u,w), w\in I(v,x)$ and $d(v,w)>\delta$. We need to show that $d(x,u)\ge d(u,v)+d(v,w)+d(w,x)-2\delta.$  Consider the three distance sums: $(A)$ $d(u,w)+d(x,v)$, $(B)$ $d(x,u)+d(w,v)$, and $(C)$ $d(x,w)+d(u,v)$. Clearly, $(A)$ is strictly larger than $(C)$.  If $(B)$ is strictly larger than $(A)$, then   $d(x,u)+d(w,v)>d(u,w)+d(x,v)=d(u,v)+2d(v,w)+d(w,x).$ On the other hand, $d(u,x)+d(w,v)\le d(u,v)+d(v,w)+d(w,x)+d(w,v)= d(u,v)+2d(v,w)+d(w,x)$. The obtained contradiction shows that $(A)$ is at least $(B)$. So, $(A)$ is the largest of the three sums. If $(C)$ is the second largest sum then, by $\delta$-hyperbolicity, $2\delta\ge (A)-(C)= d(u,w)+d(x,v)-d(x,w)-d(u,v)=2d(v,w)>2\delta$, which is impossible.  Hence, $(B)$ is the second largest sum and, by  $\delta$-hyperbolicity, $2\delta\ge (A)-(B)= d(u,w)+d(x,v)-d(x,u)-d(w,v)=d(u,v)+2d(v,w)+d(w,x)-d(x,u)-d(v,w)=d(u,v)+d(v,w)+d(w,x)-d(x,u)$, i.e., $d(x,u)\ge d(u,v)+d(v,w)+d(w,x)-2\delta.$ 
\qed
\end{proof}

We believe that the study of ($\lambda,\mu$)-bow metrics could help in deriving new properties of $\alpha_i$-metric graphs and $\delta$-hyperbolic graphs.

\section*{Data Availability Statement}

Data sharing is not applicable to this article as no datasets were generated or analysed during the current study.


\end{document}